\documentclass[a4paper]{amsart} 

\usepackage{lmodern}
\usepackage{amsmath,epsfig,mathrsfs,MnSymbol}
\usepackage{cases}
\usepackage{enumerate} 
\usepackage{cite}
\usepackage{bm}
\usepackage{bbm}
\usepackage{comment} 
\usepackage{url} 
\usepackage[english]{babel}

\usepackage{color}
\definecolor{red}{rgb}{0.75,0,0}
\definecolor{green}{rgb}{0,0.5,0}
\definecolor{blue}{rgb}{0,0,0.75}
\usepackage[colorlinks=true, pdfstartview=FitV,
urlcolor=blue,       
filecolor=green,     
citecolor=green,     
linkcolor=red,       
linktocpage,pdfpagelabels, bookmarksnumbered,bookmarksopen] 
{hyperref}

\DeclareFontFamily{OT1}{pzc}{}
\DeclareFontShape{OT1}{pzc}{m}{it}{<-> s * [1.10] pzcmi7t}{}
\DeclareMathAlphabet{\mathpzc}{OT1}{pzc}{m}{it}

\usepackage{everysel}
\EverySelectfont{%
	\fontdimen2\font=0.4em
	\fontdimen3\font=0.2em
	\fontdimen4\font=0.1em
	\fontdimen7\font=0.1em
	\hyphenchar\font=`\-
}

\newtheorem{theorem}{Theorem}[section]
\newtheorem{prop}[theorem]{Proposition}

\newtheorem{lemma}[theorem]{Lemma}
\newtheorem{defn}[theorem]{Definition}

\numberwithin{equation}{section}

\def\dys{\displaystyle}
\def\vs{\vspace{1mm}}
\def\dd{\snr{\,\cdot\,}_{\h}}
\def \eps{{\varepsilon}}
\def \r{{\mathbb{R}}}
\def \h{{\mathbb{H}^n}}
\def \N{{\mathbb{N}}}
\def\gamp{\Gamma^{+}}
\def\ue{u_{\eps}}
\def\mue{\mu_{\eps}}
\def\mea{\mathcal{M}(\overline{\Omega})}

\def\XXint#1#2#3{{\setbox0=\hbox{$#1{#2#3}{\int}$}
		\vcenter{\hbox{$#2#3$}}\kern-.5\wd0}}

\newcommand{\mbu}{\bar{\mu}_{\eps}}
\newcommand{\El}{\mathcal{E}_\lambda}
\newcommand{\Es}{\mathcal{E}^*}
\newcommand{\Fep}{\mathcal{F}_{\eps}}
\newcommand{\Fc}{\mathcal{F}}
\newcommand{\Sc}{{S}^1_0}
\newcommand{\Xc}{{X}}

\newcommand{\Ssub}{\Sob_\eps}
\newcommand{\Som}{\Sob_\Om}
\newcommand{\Sob}{S^\ast}
\newcommand{\uz}{u^{(0)}}
\newcommand{\lambdakj}{{\lambda_k^{(j)}}}
\newcommand{\xikj}{{\xi_k^{(j)}}}
\newcommand{\ukj}{{u_k^{(j)}}}
\newcommand{\rkj}{{r_k^{(j)}}}
\newcommand{\uj}{{u^{(j)}}}
\newcommand{\rr}{\rho}
\newcommand{\snr}[1]{\lvert #1\rvert}

\newcommand{\tows}{\stackrel{\ast}{\rightharpoonup}}

\newcommand{\tow}{\rightharpoonup}
\newcommand{\towt}{\stackrel{\mathpzc{t}}{\rightarrow}}

\newcommand{\Om}{\Omega}
\newcommand{\Omb}{\overline{\Omega}}
\newcommand{\Irm}{\text{I}}

\renewcommand{\rho}{\varrho}

\title[Critical Sobolev embedding in the Heisenberg group]{Struwe's Global Compactness \\ and energy approximation of the critical Sobolev embedding in the Heisenberg group}

\author[G. Palatucci]{Giampiero Palatucci}  \address{Giampiero Palatucci\\Dipartimento di Scienze Matematiche, Fisiche e Informatiche, Universit\`a di Parma\\ Parco Area delle Scienze 53/a, Campus, 43124 Parma, Italy} \email{\url{giampiero.palatucci@unipr.it}}
\author[M. Piccinini]{Mirco Piccinini}  \address{Mirco Piccinini\\Dipartimento di Scienze Matematiche, Fisiche e Informatiche, Universit\`a di Parma\\ Parco Area delle Scienze 53/a, Campus, 43124 Parma, Italy; \and Dipartimento di Matematica, Universit\`a di Pisa \\ L.go B. Pontecorvo, 5, 56127, Pisa, Italy}
\email{\url{mirco.piccinini@dm.unipi.it}}
\author[L. Temperini]{Letizia Temperini}  \address{{\color{black}Letizia Temperini\\Dipartimento di Ingegneria Industriale e Scienze Matematiche, Universit\`a Politecnica delle Marche\\ Via Brecce Bianche, 12, 60131 Ancona, Italy}} \email{\url{l.temperini@staff.univpm.it}}

\begin{document}
	
\makeatletter
\@namedef{subjclassname@2020}{\textup{2020} Mathematics Subject Classification}
\makeatother
\subjclass[2020]{35R03, 46E35, 35J08, 35A15\vspace{0.9mm}}
	
\keywords{Sobolev embeddings,  Heisenberg group, Profile decompositions, CR Yamabe, Global compactness\vspace{0.9mm}}

	\thanks{{\it Aknowledgements.}
     The authors are also supported by INdAM projects ``Fenomeni non locali in problemi locali", CUP\_E55F22000270001.  The first two authors are supported by ``Problemi non locali: teoria cinetica e non uniforme ellitticit\`a'', CUP\_E53C220019320001 and ``Problemi ellittici e sub-ellittici: singolarit\`a e crescita critica'', CUP\_E53C23001670001. The second author is also supported by the Project ``Local vs Nonlocal: mixed type operators and nonuniform ellipticity", CUP\_D91B21005370003.
 \\    The results in this paper have been announced in the preliminary research report~\cite{PPT23}}

	\maketitle

	\begin{abstract}
		 We investigate some of the effects of the lack of compactness in the critical Folland-Stein-Sobolev embedding in very general (possible non-smooth) domains,
       by proving via De Giorgi's $\Gamma$-convergence techniques that optimal functions for a natural subcritical approximations of the Sobolev quotient concentrate energy at one point. In the second part of the paper, we try to restore the compactness by extending the celebrated Global Compactness result to the Heisenberg group via a completely different approach with respect to the original one by Struwe~\cite{Str84}.
\end{abstract}

	%
	%

\setcounter{equation}{0}\setcounter{theorem}{0}

\maketitle
\vspace{5mm}

    \section{Introduction}

The aim of this paper is to investigate some of the effects of the lack of compactness in the critical Folland-Stein-Sobolev embedding in general domains.
Let  $\h := (\mathbb{R}^{2n}\times\mathbb{R},\circ, \delta_\lambda)$ denote the  Heisenberg-Weyl group, and
consider the standard Folland-Stein-Sobolev space~$\Sc(\h)$ defined as the completion of~$C^\infty_0(\h)$ with respect to the homogeneous subgradient norm~$\|D_H\cdot\|_{L^2(\h)}$.  As well known, the following Sobolev-type inequality, first proved in~\cite{FS74}, asserts that for a positive constant~$\Sob$ it holds
\begin{equation}\label{folland}
\|u\|^{2^\ast}_{L^{2^\ast}(\h)} \leq \Sob \|D_Hu\|^{2^\ast}_{L^2(\h)} \quad \forall u \in S^1_0(\h)\,,
\end{equation}
where we denote by $2^\ast=2^\ast(Q):=2Q/(Q-2)$ the Folland-Stein-Sobolev critical exponent which depends on the {\it homogeneous dimension}~$Q:=2n+2$ of the Heisenberg group~$\h$.
\vspace{2mm}

Over the past decades, the critical Sobolev inequality~\eqref{folland} has garnered significant attention as a captivating subject of study. Given the extent of the literature, providing an exhaustive treatment here is not feasible. We just mention the highly significant papers~\cite{JL88,GL92,LU98,GV00,CU01}, the recent book~\cite{IV11}, and the references therein for in-depth insights.
The main reasons for studying the critical Sobolev inequality lie in its close association with the lack of compactness of the related critical Folland-Stein-Sobolev embedding and its connection to the corresponding Euler-Lagrange equation. This, in turn, addresses the significant CR Yamabe problem. 
Numerous results have been established, indicating compliance with the classical critical inequality in the Euclidean framework, despite the complexities posed by the sub-Riemannian geometry of the Heisenberg setting; {see for instance  the aforementioned list of papers above as well as the very recent ones~\cite{CLMR23,FV23,PV24}, where the solutions to the critical Yamabe equation in $\h$ are classified, and the references therein}. However, many expected results remain open due to substantial differences from the Euclidean framework, stemming from the complex metric structure of the Heisenberg setting and the presence of characteristic points in the domains involved.

\vspace{2mm}

The contribution of the present paper is twofold.
Firstly, we investigate a natural approximation to~\eqref{folland} via subcritical Sobolev inequalities. Our approach is variational. For this, it is convenient to consider the following maximization problem,
\begin{equation}\label{critica0}
\Sob:=\sup\left\{\int_{\h}|u(\xi)|^{2^\ast}\,{\rm d}\xi \, : \, u\in \Sc(\h), \int_{\h}|D_H u(\xi)|^2{\rm d}\xi \leq 1\right\}.
\end{equation}

Clearly, the validity of~\eqref{folland} is equivalent to the finiteness of the constant~$\Sob$, defined in~\eqref{critica0}.
The existence of maximizers in~\eqref{critica0} poses a challenging problem due to the intrinsic dilations and translations invariance of this inequality, analogous to the difficulties encountered in the classical critical Sobolev inequality. However, the situation becomes even more delicate in this context, given the underlying non-Euclidean geometry of the Heisenberg group and the obstacles arising from non-commutativity. The groundbreaking paper by Jerison and Lee~\cite{JL88} has provided explicit forms of the maximizers, along with the computation of the optimal constant in~\eqref{critica0}. We also refer to the fundamental paper~\cite{FL12} where sharp constants for inequalities on the Heisenberg group have been derived for even more general cases, in turn obtaining sharp constants for the corresponding duals, which are the Sobolev inequalities for the sub-Laplacian and the conformal fractional Laplacians.  
Additionally, related investigations on the Moser-Trudinger inequality in the sub-Riemannian setting can be found in~\cite{LL12,BFM13}, and also in~\cite{FL12} via a limiting process and duality (the logarithmic Hardy-Littlewood-Sobolev inequality).

\vspace{2mm}

Similarly, for any bounded domain~$\Omega\subset\h$, we consider the following Sobolev embedding in the same variational form as the one in~\eqref{critica0},
\begin{equation}\label{critica}
\Som:=\sup\left\{\int_{\Omega}|u(\xi)|^{2^\ast}\,{\rm d}\xi \, : \, u\in \Sc(\Omega), \int_{\Omega}|D_H u(\xi)|^2{\rm d}\xi \leq 1\right\},
\end{equation}
where now the Folland-Stein-Sobolev space $\Sc(\Omega)$ is given as the closure of $C^\infty_0(\Omega)$ with respect to the homogeneous $L^2$-subgradient norm in $\Omega$. It can be verified that $\Som\equiv \Sob$ using a standard scaling argument on compactly supported smooth functions. However, in contrast to the critical case presented in \eqref{critica0}, the variational problem \eqref{critica} has no maximizers, as evidenced by the explicit form of the optimal functions in \eqref{critica0} (see, for instance, Theorem \ref{thm_optimal}). The scenario changes drastically for the subcritical embeddings. In this case, due to the boundedness of $\Omega$, the embedding $\Sc(\Omega) \hookrightarrow L^{2^\ast-\eps}(\Omega)$ becomes compact, given that $\eps<2^\ast-2$, ensuring the existence of a maximizer $u_\eps\in \Sc(\Omega)$ for the associated variational problem
   \begin{equation}\label{sobolev}
\Sob_\eps:=\sup\left\{\int_{\Omega}|u(\xi)|^{2^\ast-\eps}\,{\rm d}\xi \, : \, u\in \Sc(\Omega), \int_{\Omega}|D_H u(\xi)|^2{\rm d}\xi \leq 1\right\}.
\end{equation}

In the same fashion, in the Euler-Lagrange equation for the energy functionals in~\eqref{sobolev}, that is
\begin{equation}\label{equazione}
-\Delta_H u_\eps = \lambda |u_\eps|^{2^\ast-\eps-2} u_\eps \, \ \text{in} \ (\Sc(\Omega))',
\end{equation}
where $\lambda$ is a suitable Lagrange multiplier, we spot a different behaviour for $\eps>0$  and $\eps=0$ . Whereas when $\eps>0$ the problem above has a solution~$u_\eps$, it becomes very delicate when $\eps=0$: even the existence of solutions is not granted. 
{ In particular, the existence and various properties of the solutions strongly depend on both the geometry and the topology of the domain~$\Omega$. We refer to the following papers: 
\cite{LU98} for nonexistence of nonnegative solutions when~$\Om$ is a particular half-space of $\h$ with boundary parallel to the center $\{(0,0,t)\,|\, t\in\mathbb{R}\}$ of $\h$ ; \cite{U99} for the extension to general half-spaces transverse to the center of $\h$;  \cite{GL92} for existence and nonexistence results for even more general nonlinearity;} \cite{CU01}~where the authors show the existence of a solution in the case when the domain~$\Omega$ has at least a nontrivial suitable homology group.

\vspace{2mm}

In view of such a qualitative change when $\eps=0$ in both~\eqref{sobolev} and \eqref{equazione}, in the present paper we analyze the asymptotic behaviour as $\eps$ goes to $0$ of both
the subcritical Sobolev constant~$\Sob_\eps$ in the Heisenberg group given in~\eqref{sobolev} and of the corresponding optimal functions~$u_\eps$ of the embedding $\Sc(\Omega) \hookrightarrow L^{2^\ast-\eps}(\Omega)$.

\vspace{2mm}
Regarding {\it the~Euclidean~counterpart} of this investigation, several results have been obtained, mostly via fine estimates and a standard  elliptic regularity approach. These investigations primarily focus on a special class of solutions of equation~\eqref{equazione}, which serve as maximizers for the corresponding Sobolev embedding. Energy concentration results for these sequences, along with the subsequent localization of such concentration on specific points, have been established in~\cite{BP89,Rey89,Han91}, among other works.
\\
However, the situation is quite different in {\it the~Heisenberg~panorama} due to the numerous inherent challenges in this framework. Indeed, the sub-Riemannian geometry prevents a straightforward generalization of several tools and symmetrization techniques, as well as regularity approximations. As a result, the Heisenberg setting remains relatively unexplored in comparison, leaving the field with significant research opportunities. It is worth mentioning a first important result in the recent paper~\cite{MMP13}, where the authors are able to construct a concentrating sequence of solutions to~\eqref{equazione} for certain non-degenerate critical point of  the regular part of the Green function of~$\Om$.

\vspace{2mm}

 Our first main result is the subcritical approximation of the Sobolev quotient~$S^\ast$ in the Heisenberg group  described
  below.
    \begin{theorem}\label{the_gamma-intro}
Let $\Omega\subseteq\h$ be a bounded domain, and denote by~$\mea$ the set of nonnegative Radon
measures in~$\Omega$. Let $\Xc=\Xc(\Omega)$ be the space
\[
\Xc:=\Big\{(u,\mu) \in \Sc(\Omega)\times\mea: \mu \geq |D_H u|^{2}{\rm d}\xi, \, \mu(\Omb)\leq 1\Big\},
\]
endowed with the product  topology~$\mathpzc{t}$ such that
\begin{equation}\label{def_top1}
(u_k,\mu_{k}) \towt (u,\mu) \ \, {\stackrel{\text{def}}{\Leftrightarrow}} \ \, \begin{cases} u_{k} \rightharpoonup u \ \text{in} \ L^{2^{\ast}}\!(\Omega), \\
\mu_{k} \tows \mu \ \text{in} \ \mea.
\end{cases}
\end{equation}
\noindent
\\ Let us consider the following family of functionals,
\begin{equation}\label{def_fue1}
\Fep(u,\mu):= \dys \int_{\Omega}|u|^{2^{\ast}\!-\eps} {\rm d}\xi  \ \ \ \forall (u,\mu) \in \Xc\,.
\end{equation}
Then, as $\eps\to 0$, the $\Gamma^{+}$-limit of the family of functionals $\Fep$ with respect to the topology~${\mathpzc{t}}$ given by~{\rm(\ref{def_top1})} is the functional $\Fc$ defined by
\begin{equation}\label{def_fu}
\Fc(u,\mu)=\int_{\Omega}|u|^{2^{\ast}}{\rm d}\xi + S^{\ast}\sum_{j=1}^{\infty}\mu_{j}^{\frac{2^{\ast}}{2}} \ \ \ \forall (u,\mu) \in \Xc.
\end{equation}
Here $S^{\ast}$ is the best Sobolev constant in $\h$, $2^\ast=2Q/(Q-2)$ is the Folland-Stein-Sobolev critical exponent,
and the numbers $\mu_j$ are the coefficients of the atomic part of the measure $\mu$.
\end{theorem}

In order to prove such a result in the most general situation -- thus requiring no additional regularity assumptions nor special geometric features on the domains --  we attack the problem pursuing a new approach and for this we rely on De Giorgi's $\Gamma$-convergence techniques. This is in the same spirit of previous results regarding the classical Sobolev embedding in the Euclidean framework, as seen in~\cite{AG03,Pal11,Pal11b}.
However, our main proof deviates significantly, as we explicitly construct the optimal recovery sequences: a departure from the implicit existence results demonstrated in the aforementioned papers, which rely on compactness and locality properties of the $\Gamma$-limit energy functional.
 Interestingly, our strategy bears resemblance to the
  {\it fractional Sobolev spaces\,} scenario in \cite{PPS15}, although notable differences inevitably arise due to the natural discrepancy
between the involved frameworks.

\vspace{2mm}

    As a corollary of Theorem~\ref{the_gamma-intro} above, we can deduce that the sequences of maximizers~$\{u_\eps\}$ for the subcritical Sobolev quotient~$\Sob_\eps$ do concentrate energy at one point~$\xi_{\rm o}\in\Omb$, in clear accordance
  with the analogous result in the Euclidean case.

   \vspace{2mm}
  \begin{theorem}\label{cor_concentration}
  	Let $\Om \subset \h$ be a bounded domain and
	let~$\ue\in \Sc(\Omega)$ be a maximizer for~$S^{\ast}_{\eps}$.  Then, as~$\varepsilon=\varepsilon_k \to 0$, up to subsequences, we have that
  	there exists $\xi_{\rm o} \in \Omb$ such that
  	\[
  	u_k=u_{\eps_k}  \rightharpoonup 0  \ \text{in} \ L^{2^\ast}\!(\Omega),
  	\]
  	and
  	\[
  	\dys  |D_H u_k|^{2}{\rm d}\xi \tows \delta_{\xi_{\rm o}} \ \text{in} \ \mea,
  	\]
  	with~$\delta_{\xi_{\rm o}}$ being the Dirac mass at~$\xi_{\rm o}$.
  \end{theorem}
The concentration result above in very general domains was in fact one of our first goals, as well as the veritable guideline to the choice of the topology in the $\Gamma^+$-convergence Theorem~\ref{the_gamma-intro} in order to provide a convergence of functionals which guarantees that maximizers of the approximating
functionals converge to maximizers of a precise limit functional. Indeed, at this stage the proof of Theorem~\ref{cor_concentration} does reduce to a fine analysis of the form of the maxima of the functional~$\mathcal{F}$ defined by~\eqref{def_fue1}; see~Section~\ref{sec_concentration}. For further related results, we refer to the very recent paper~\cite{PP23h}, where two of the authors were able to localize the concentration point in terms of the Green function associated to the domain~$\Om$.

\vspace{2mm}

    \vspace{2mm}
    In the second part of the present paper, we investigate a way to circumvent the lack of compactness in the critical Folland-Stein-Sobolev embedding, by proving the validity of the so-called Global Compactness in the Heisenberg framework.

    Since the seminal paper~\cite{Str84} by Struwe, the celebrated Global Compactness in the Sobolev space $H^1$ has become a fundamental tool in Analysis. It has played a crucial role in numerous existence results, such as e.~\!g. for ground states solutions for nonlinear Schr\"odinger equations, for solutions of Yamabe-type equations in conformal geometry, for prescribing $Q$-curvature problems, harmonic maps from Riemann surfaces into Riemannian manifolds, Yang-Mills connections over four-manifolds, planar Toda systems, etc. The involved literature is really too wide to attempt any reasonable account here. In Theorem~\ref{thm_glob_comp} below, we state the counterpart of Struwe's Global Compactness in the Heisenberg group setting.

\vspace{2mm}
In order to precisely state our result, consider  for any fixed~$\lambda\in\r$ the problem,
\begin{equation}\label{plambda}
	-\Delta_Hu-\lambda u-|u|^{2^*-2}u=0\qquad\text{in } (\Sc(\Omega))', \tag{$P_\lambda$}
\end{equation}
together with its corresponding Euler--Lagrange functional~$\El:\Sc(\Omega)\to\r$ given by
\begin{equation*}
	\El(u) =\frac12 \int_{\Om}|D_H u|^2 	\,{\rm d}\xi -\frac{\lambda}{2}\int_{\Om}|u|^2	\,{\rm d}\xi-\frac{1}{2^*}\int_{\Om}|u|^{2^*}	\,{\rm d}\xi.
\end{equation*}
Consider  also the following limiting problem,
\begin{equation}\label{pzero}
	-\Delta_Hu-|u|^{2^*-2}u=0\qquad\text{in } (\Sc(\Om_{\rm o}))',\tag{$P_0$}
\end{equation}
where $\Om_{\rm o}$ is either a half-space or the whole~$\h$;
i.~\!e., the Euler-Lagrange equation corresponding to the energy functional~$\Es: \Sc(\Om_{\rm o})\to\r$,
\begin{equation*}
	\Es(u)=\frac12 \int_{\Om_{\rm o}}|D_H u|^2 	\,{\rm d}\xi -\frac{1}{2^*}\int_{\Om_{\rm o}}|u|^{2^*}	\,{\rm d}\xi.
\end{equation*}

{Before stating the next result, we just recall that an half-space of $\h$ is simply an half-space of $\mathbb{R}^{2n+1}$, according to the definition in \cite{LU98}, see also \cite{CU01}.}
We have the following
\begin{theorem}[\bf Global Compactness in the Heisenberg group]\label{thm_glob_comp}
\text{}\\	Let~$\{u_k\}\subset \Sc(\Omega)$ be a Palais-Smale sequence for~$\El$; i.~\!e., such that
	\begin{eqnarray}
		&&\El(u_k)\leq c\quad \text{for all }k,\label{PS1}\\*
		&&{\rm d}\El(u_k) \rightarrow 0\quad \text{as } k\to\infty \quad \text{in }(\Sc(\Omega))'\label{PS2}.
	\end{eqnarray}
	Then, there exists a (possibly trivial) solution~$\uz\in \Sc(\Omega)$ to~\eqref{plambda} such that, up to a subsequences, we have
	\[
	u_k\rightharpoonup\uz\quad \text{as } k\to\infty \quad\text{in }\Sc(\Omega).
	\]
	Moreover, either the convergence is strong or there is a finite set of indexes~$\Irm=\{1,\dots,J\}$ such that for all~$j\in\Irm$ there exist a
	nontrivial solution~$\uj\in \Sc(\Om_{\rm o}^{(j)})$ to~\eqref{pzero} with $\Om_{\rm o}^{(j)}$ being either a half-space or the whole~$\h$,
	a sequence of nonnegative numbers~$\{\lambdakj\}$ converging to zero and a sequences of points~$\{\xikj\}\subset\Om$ such that, for a renumbered subsequence, we have for any~$j\in\Irm$
	\[
	\ukj(\cdot):=\lambdakj^{\frac{Q-2}{2}}u_k\big(\tau_{\xikj}\big(\delta_{\lambdakj}(\cdot)\big)\big) \rightharpoonup \uj(\cdot)\quad\text{in }\Sc(\h) \quad\text{ as }k \to\infty.
	\]
	In addition, as~$k\to\infty$ we have
	\begin{eqnarray}
		&&u_k(\cdot)=\uz(\cdot)+\sum_{j=1}^{J}\lambdakj^{\frac{2-Q}{2}}u_k\big(\delta_{1/\lambdakj}\big(\tau_{\xikj}^{-1}(\cdot)\big)\big)+o(1) \quad\text{ in }\Sc(\h);\label{propr1}\\*
		&&\left|\log{\frac{\lambda_k^{(i)}}{\lambdakj}}\right|+\left|\delta_{1/\lambdakj}\big(\xikj^{-1}\circ \xi_k^{(i)}\big)\right|_{\h}\to\infty\quad\text{for }i\neq j,\ \,i,j\in\Irm;\label{propr2}\\*
		&&\|u_k\|_{\Sc}^2=\sum_{j=0}^{J}\|\uj\|_{\Sc}^2+o(1);\label{propr3}\\*
		&&\El(u_k)=\El(\uz)+\sum_{j=1}^{J}\Es(\uj)+o(1)\label{propr4}.
	\end{eqnarray}
\end{theorem}

The original proof by Struwe in~\cite{Str84} consists of a subtle analysis concerning how the Palais-Smale condition does fail for the functional~$\mathcal{E}^\ast$, based on rescaling arguments, used in an iterated way to extract convergent subsequences with nontrivial limit, together with some slicing and extension procedures on the sequence of approximate solutions to~\eqref{plambda}. Such a proof revealed to be very difficult to extend to different frameworks, and the aforementioned strategy seems even more cumbersome to be adapted to the Heisenberg framework considered here. {However, several remarkable investigations regarding the behaviour of Palais-Smale sequences for the critical energy~$\Es(\cdot)$ can be still found in~\cite{Cit95} where the author proves an analogous representation result for {nonnegative Palais-Smale sequences} in order to prove existence results of positive solutions to a class of semilinear Dirichlet problem involving critical growth, in the same spirit of the well-known Brezis-Nirenberg result for the classical Laplacian. 
Moreover, it is worth mentioning the relevant paper~\cite{GMM18}, where the authors deduce the desired Global Compactness in the important case of critical energies on the $(2n+1)$-dimensional sphere (equipped with the CR structure) associated to the sub-elliptic intertwining operator~$\mathcal{L}_{2\kappa}$ of order~$2\kappa$, with~$\kappa \in \mathbb{R}$ being such that~$0 < 2\kappa < Q$; also covering fractional CR Yamabe energies.

In the present manuscript, we completely changes {the} approach to the problem, and we will be able to show how to deduce the Global Compactness result in Theorem~\ref{thm_glob_comp} for very general domains~$\Omega\subset \h$ in quite a simple way by means of the
so-called {\it Profile Decomposition}, firstly proven by  G\'erard in~\cite{Ger98} for bounded sequences in the fractional (Euclidean) space~${H}^s$.}

\vspace{3mm}

{\it To summarize, the contribution of the present paper} is twofold: we investigate  the De Giorgi's $\Gamma$-convergence energy approximation of the critical Folland-Stein-Sobolev embedding in turn implying an expected concentration result in very general (possibly non smooth) domains; we extend the Global Compactness to the Heisenberg group framework.

\vspace{2mm}
\subsection{Related open problems and further developments}

Starting from the results proven in the present paper, several questions naturally arise.
	\\*[0.5ex]
	\indent $\bullet$~~One can consider to investigate the fractional counterpart of the results proven here; that is, by replacing the $S^1_0$-norm in~\eqref{folland} by the~$S^s_0$-norm of differentiability order~$s\in(0,1)$.  Some of the main tools developed in the present paper and the general approach in the proof of the $\Gamma$-convergence theorem appear to be repeatable at some extents; as well as some of the tools in order to achieve the concentration result; e.~g., the very general Profile Decomposition theorem by G\'erard which is natively presented in the fractional framework. However, most of the proofs in order to get the precise localization result would require a completely new approach because of the nonlocality of the fractional sub-Laplacian operator. In this respect, one should deal with truncations and the resulting error term in the same flavour of the papers~\cite{PP22,MPPP23,PP23}, where a precise quantity, the so-called ``nonlocal tail'', has been firstly introduced to attack very general equations led by fractional sub-Laplacian-type operators.  On the same line of thought, important related results can be find in~\cite{GMM18}. 

Otherwise, quite a different energy approach in the nonlocal framework could be carried out via an auxiliary harmonic extension problem to the Siegel upper half-space, by taking into account that conformally invariant fractional powers of the sub-Laplacian on the Heisenberg group can be given in terms of the scattering operator, as seen in the relevant paper~\cite{FGMT15}.

	\vspace{1.3mm}

$\bullet$~~It could be interesting to study the nonlinear counterpart of both the local and the nonlocal Heisenberg framework. We believe that some of the results could be extended in the same spirit of the Euclidean counterpart, as for instance the $\Gamma$-convergence subcritical approximation scheme
and the related concentration result; see~\cite{Pal11b}. In order to do this, one would need to change the approach to the $\Gamma$-liminf proof by possibly making use of the precious results regarding the lack of compactness for the related $p$-sub-Laplacian energy in~\cite{GV00,Vas06}, and some very recent result regarding the $(s,p)$-sub-Laplacian (as for instance the aforementioned ``nonlocal tail'' approach in~\cite{PP22}), respectively.

	\vspace{1.3mm}
	
		$\bullet$~~A more challenging extension could be that in the very general $H$-type groups setting.
Within this framework, numerous expected results, particularly concerning important properties of the associated extremal functions, remain largely unexplored. A starting point could be the investigation of their subclass of groups of Iwasawa type. For this, one can take advantage of the involved group structure, as well as of important results present in the literature; that is, the investigation
 in~\cite{GV01}, where positive solutions to the CR~Yamabe equations being invariant with respect to the action of the orthogonal group in the first layer of the Lie algebra have been precisely characterized.

	\vspace{1.3mm}
	
	$\bullet$~~It is worth stressing that the limiting domain~$\Om_{\rm o}$ in the Global Compactness Theorem~\ref{thm_glob_comp}  can be either the whole~$\h$ or a half-space. On the contrary, in the original proof in the Euclidean case by~Struwe in~\cite{Str84} one can exclude the existence of nontrivial solutions to the limiting problems in the half-space by Pohozaev identity and unique continuation. This is still the case when dealing with {\it nonnegative} solutions in the fractional framework~$H^s_0$; see~\cite{PP15}. However, in all the remaining Euclidean cases (as, e.~g., in $W^{s,p}$, in $H^2_0$, and so on), this possibility cannot be a priori excluded. The same happens in the sub-Riemannian setting, even in the very special case when a complete characterization of the limiting sets is possible under further regularity assumptions on~$\Omega$ (see~forthcoming Lemma~\ref{citti_uguzzoni}). In such a Heisenberg framework, a very few nonexistence results are known, basically only in the case when the domain reduces to a~half-plane parallel or {transverse} to the group center; see~\cite{LU98,U99,CU01}. Further nonexistence results are known in the class of groups of Iwasawa-type (see~\cite{GV00}).

	\vspace{1.3mm}

		$\bullet$~~Related concentration phenomena could be investigated in the sub-Riemannian setting by considering the second critical exponent, see~\cite{Pas93}, in the same spirit of~\cite{DPMP10,DMM19}, where the authors are able to prove the existence of solutions whose energy concentrates to a Dirac measure of  given geodesics of the boundary of domains with negative inner normal curvature.

	\vspace{1.3mm}

		$\bullet$~~The critical~$L^{2^\ast}$-energy in~\eqref{critica0} could be replaced by an even more general nonconvex and discontinuous energy with critical growth, as in the Euclidean framework studied in~\cite{FM99}, in clear accordance with certain
	 free-boundary problems; in the sub-Riemannian framework, this will lead to the class of problems carefully studied in~\cite{DGP07}.
	For what concerns our approach to the $\Gamma$-convergence result, the techniques appear to be adaptable to some extent, but surely they cannot be repeated in order to achieve the $\Gamma$-liminf inequality because of the lack of the explicit form of the involved extremal functions. For a new proof, one can possibly develop a suitable strategy by density and compactness in clear accordance with the strategy in~\cite{AG03,Pal11}.

\vspace{2mm}
\subsubsection*{{\rm 1.2.}~\bf The paper is organized as follows} In Section~\ref{sec_preliminaries} below we briefly fix the notation. In Section~\ref{sec_subcritical} we perform the analysis of the Sobolev embedding via the $\Gamma$-convergence approximations, which we can apply in Section~\ref{sec_concentration} in order to get the expected concentration results in very general domains. Section~\ref{sec_struwe} is devoted to the proof of the Global Compactness in the Heisenberg group.

  \vspace{2mm}
   \section{Preliminaries}\label{sec_preliminaries}
   In this section, we clarify the notation used throughout the paper by briefly revisiting essential properties of the Heisenberg group. Additionally, we present some well-known results concerning the lack of compactness in the critical Sobolev embedding within the Folland-Stein spaces of the Heisenberg group.

   \vspace{2mm}
   \subsection{The Heisenberg-Weyl group}\label{sec_heis}
    We start by summarily recalling a few well-known facts about the Heisenberg group.
\vspace{1mm}

We denote points~$\xi$ in~$\mathbb{C}^n\times\r \simeq \r^{2n+1}$ by
   \[
   \xi := (z,t) = (x+iy, t) \simeq (x_1,\dots,x_n, y_1,\dots,y_n,t)
   \in \r^n\times \r^n \times \r.
   \]
   For any~$\xi,\xi'\in \r^{2n+1}$,  {\it the group multiplication law}~$\circ$ is defined by
   \[
   	\xi \circ \xi' := \Big(x+x',\, y+y',\, t+t'+2\langle y,x'\rangle-2\langle x,y'\rangle \Big).
  \]
 Given~$\xi'\in\h$, {\it the left translation}~$\tau_{\xi'}$ is defined by
\begin{equation}\label{def_tau}
 \tau_{\xi'}(\xi):=\xi'\circ\xi \qquad \forall \xi \in \h.
 \end{equation}
The group of non-isotropic {\it dilations}~$\{\delta_\lambda\}_{\lambda>0}$ on~$\r^{2n+1}$ is defined by
 \begin{equation}\label{def_philambda}
   	                 \xi   \mapsto \delta_{\lambda}(\xi):=(\lambda x,\, \lambda y,\, \lambda^2 t),
	               \end{equation}
	                  and, 	
	                 as customary,
   $Q\equiv 2n+2$ is the {\it homogeneous dimension of\,}~$\r^{2n+1}$ with respect to~$\{\delta_\lambda\}_{\lambda>0}$,
  so  that the Heisenberg-Weyl group~$\h := (\r^{2n+1},\circ, \delta_\lambda)$ is a homogeneous Lie group.
\vs

   The Jacobian base of the Heisenberg Lie algebra~$\mathscr{H}^n$ is given by
   \[
   Z_j := \partial_{x_j} +2y_j\partial_t, \quad Z_{n+j}:= \partial_{y_j}-2x_j\partial_t, \quad 1 \leq j\leq n, \quad T:=\partial_t.
   \]
   Since
   $[Z_j,Z_{n+j}]=-4\partial_t$ for every  $1 \leq j \leq n$, it plainly follows that
   \begin{eqnarray*}
   	&& \textup{rank}\Big(\textup{Lie}\{Z_1,\dots,Z_{2n},T\}(0,0)\Big)
   	\ = \ 2n+1,
   \end{eqnarray*}
   so that~$\h$ is a Carnot group with the following stratification of the algebra
   \[
 \mathscr{H}^n  = \textup{span}\{Z_1,\dots,Z_{2n}\} \oplus \textup{span}\{T\}.
   \]
   The horizontal (or intrinsic) gradient~$D_H$ of the group is given by
   \[
   D_H u(\xi) := \left( Z_1 u(\xi),\dots, Z_{2n}u(\xi)\right).
   \]
   The Kohn Laplacian (or sub-Laplacian)~$\Delta_H$ on~$\h$ is the second order operator invariant with respect to the left-translations~$\tau_{\xi'}$ defined in~\eqref{def_tau} and homogeneous of degree~$2$ with respect to the dilations~$\delta_\lambda$ defined in~\eqref{def_philambda},
   \[
   \Delta_H :=  \sum_{j=1}^{2n} Z^2_j.
   \]

    \vs
	A {\it homogeneous norm} on~$\h$ is a continuous function (with respect to the Euclidean topology)~$\dd : \h \rightarrow [0,+\infty)$ such that:
	\begin{enumerate}[\qquad\qquad\ \,   \rm(i)]
		\item{
			$ \snr{\delta_\lambda(\xi)}_{\h}=\lambda \snr{\xi}_{\h}$, for every $\lambda>0$ and every $\xi \in \h$;
		}\vspace{1mm}
		\item{
			$\snr{\xi}_{\h}=0$ if and only if $\xi=0$.}
	\end{enumerate}
We say that the homogeneous norm~$\dd$ is {\it symmetric} if
	$
    \snr{\xi^{-1}}_{\h}= \snr{\xi}_{\h}$ for all $\xi \in \h$.
     If~$\dd$ is a homogeneous norm on~$\h$, then the function
	$
	(\xi,\eta)\mapsto \snr{\eta^{-1}\circ \xi}_{\h}$
	is a pseudometric on~$\h$. In particular, we will work with the standard homogeneous norm on~$\h$, also known as {\it Kor\'anyi gauge},
    \begin{equation*}
	|\xi|_{\h} := \left(|z|^4 +t^2\right)^\frac{1}{4} \qquad \forall \xi=(z,t) \in \h.
    \end{equation*}
    As customary, we will denote by~$B_\rr\equiv B_\rr(\eta_{\rm o})$ the ball with center~$\eta_{\rm o} \in \h$ and a radius~$\rr>0$ given by
    \[
    B_\rr(\eta_{\rm o}):=\Big\{\xi \in \h : |\eta_{\rm o}^{-1}\circ \xi|_{\h} < \rr\Big\}.
    \]

\vspace{2mm}
\subsection{Lack of compactness in the critical Sobolev embedding}\label{sec_cca}
In this subsection, we revisit crucial results within the Heisenberg framework, focusing on the analysis of the impact of the lack of compactness in the critical Sobolev embedding.

\vspace{1mm}

Firstly, we state (in the form adapted to our framework) the aforementioned pioneering result by Jerison and Lee~\cite{JL88} which gives  the optimal constant in the critical Sobolev inequality  together with the explicit expression of the functions giving the equality in~\eqref{critica0}.
\begin{theorem}[Corollary~C in~\cite{JL88}]\label{thm_optimal}
Let $2^\ast=2Q/(Q-2)$. Then for any $\lambda>0$ and any $\xi_{\rm o}\in \h$, the function~$U_{ {\lambda}, \xi_{\rm o}}$ defined by
\begin{equation}\label{talentiane}
U_{{\lambda}, \xi_{\rm o}} := U \left(\delta_{\frac{1}{\lambda}}\big(\tau_{\xi^{-1}_0}(\xi)\big)\right)\,,
\end{equation}
where
\begin{equation}\label{talentiane_2}
U(\xi)=c_0\left(\big(1+|z|^2\big)^2+t^2\right)^{-\frac{Q-2}{4}}\quad \forall \xi\in\h,
\end{equation}
is solution to the variational problem~\eqref{critica0}. 
\end{theorem}

As in the classical Euclidean setting, the Concentration-compactness alternative in the Heisenberg framework,  see Theorem~\ref{thm_cca} below,  has been shown to be crucial for analyzing the asymptotic behaviour of bounded sequences in the Folland-Stein space. We also refer to \cite{PT22} for related results.

\begin{theorem}[Lemma~1.4.5 in~\cite{IV11}]\label{thm_cca} Let $\Omega \subseteq \h$ be an open subset and let $\{u_k\}$ be a sequence in~$\Sc(\Om)$ weakly converging to $u$ as $k \to \infty$ and such that
\[
|D_H u_k|^2{\rm d}\xi \tows \mu \ \ \ \text{and} \ \ \ |u_k|^{2^{\ast}}{\rm d}\xi\tows \nu \ \ \text{in} \ \mea.
\]
Then, either $u_k \to u$ in $L^{2^{\ast}}_{\rm{loc}}(\Om)$ or  there exists a (at most countable) set of distinct points $\{\xi_j\}_{j\in J}$ and positive numbers $\{\nu_j\}_{j\in J}$ such that we have
\begin{equation*}
\nu=\ |u|^{2^{\ast}}{\rm d}\xi+\sum_{j} \nu_j \delta_{\xi_j}.
\end{equation*}
Moreover, there exist a positive measure $\tilde{\mu} \in \mea$ with {\rm supp}~$\!\tilde{\mu}~\subset~\Omb$ and positive numbers $\{\mu_j\}_{j\in J}$  such that
\begin{equation*}
\mu=|D_H u|^2{\rm d}\xi+\tilde{\mu}+\sum_{j} \mu_j \delta_{\xi_j}, \quad \nu_j \leq S^{\ast} (\mu_j)^{\!\frac{2^{\ast}}2}\, .
\end{equation*}
\end{theorem}

\vspace{2mm}
We conclude this section by presenting an expected convergence result for the Sobolev quotients.
\begin{prop}\label{pro_sstareps}
Let $\Ssub$ be defined by~\eqref{sobolev} and let $S^*$ be defined by~\eqref{critica}. Then
\begin{equation*}
\dys \lim_{\eps\to0}\Ssub=S^*.
\end{equation*}
\end{prop}
\begin{proof}
We were not able to find a precise reference to the literature in the Heisenberg setting, thus we reproduce here a classical elementary proof. Firstly, since~$\Om$ is bounded, the H\"older inequality yields
\begin{equation}\label{eq_3astast}
\dys \limsup_{\eps\to0}\Ssub \leq S^*.
\end{equation}
Indeed, take the maximizers~$\ue \in \Sc(\Om)$  for $\Ssub$;  one has
\begin{eqnarray*}
\dys \Ssub  & = & \int_\Om |\ue|^{2^*-\eps}{\rm d}\xi \ \leq \ \left(\int_\Om|\ue|^{2^*}{\rm d}\xi\right)^{\!\!\frac{2^*-\eps}{2^*}}\!|\Om|^{\frac{\eps}{2^*}} \\*[1ex]
& \leq & (S^*)^{\!\frac{2^*-\eps}{2^*}}|\Om|^{\frac{\eps}{2^*}}.
\end{eqnarray*}
It suffices to pass to the limit as $\eps$ goes to zero, and the inequality in~\eqref{eq_3astast} will follow.

The remaining inequality is a mere consequence of the pointwise convergence of
the energy functional~$\int|\cdot|^{2^\ast-\eps}$ to $\int|\cdot|^{2^\ast}$.
This is standard: for every $\upsilon>0$
 there exists $u_{\upsilon}\in \Sc(\Om)$
 such that $\|D_H u_{\upsilon}\|^2_{L^2(\Omega)}\leq 1$ and
\begin{equation}\label{eq_4ast}
\dys \int_\Om |u_{\upsilon}|^{2^\ast}{\rm d}\xi >S^*-{\upsilon}.
\end{equation}
Clearly, for such a function $u_{\upsilon}$, one has
\[
\Ssub\geq
\int_\Om |u_{\upsilon}|^{2^\ast-\eps}{\rm d}\xi.
\]
Then, combining the preeceding inequality with~\eqref{eq_4ast} and passing to the limit as $\eps$ goes to zero, we get
\begin{eqnarray*}
\dys \liminf_{\eps\to 0}\Ssub & \geq & \lim_{\eps\to 0}\dys \int_\Om |u_{\upsilon}|^{2^\ast-\eps}{\rm d}\xi
\ = \ \dys \int_\Om |u_{\upsilon}|^{2^\ast}{\rm d}\xi  \\[1ex]
& \geq & S^* -{\upsilon},
\end{eqnarray*}
which gives the desired inequality in view of the arbitrariness of ${\upsilon}$.
\end{proof}

\vspace{2mm}
\section{Proof of the $\Gamma$-convergence result in Theorem \ref{the_gamma-intro}}\label{sec_subcritical}
First, we present the functional setting in which we will perform the asymptotic analysis of the Sobolev embeddings via $\Gamma$-convergence.
\subsection{The functional setting}
Let us describe the asymptotic behaviour as $\eps$ goes to $0$ of the maximizers of the Sobolev constant~$\Ssub$ associated to the embedding $\Sc(\Om) \hookrightarrow L^{2^\ast-\eps}$.
It is convenient to restate the variational form of the problem from the introduction; that is,
\begin{equation}\label{problema}
\Sob_\eps= \sup \left\{\Fep(u) \, : \, u\in \Sc(\Omega), \int_{\Omega}|D_H u|^2{\rm d}\xi \leq 1\right\},
\end{equation}
where~$\Fep$ denotes the following family of energy functionals,
\begin{equation*}
\dys \Fep (u) := \int_{\Om} |u|^{2^{\ast}\!-\eps} {\rm d}\xi,
\end{equation*}
on the set $\dys \left\{u \in \Sc(\Omega), \, \int_{\h} |D_H u|^{2} {\rm d}\xi \leq 1 \right\}$.
\vspace{1mm}

As mentioned, the main tool will be the $\Gamma$-convergence in the sense of De Giorgi. For this, it is crucial to
introduce a convenient functional setting, that is, the functional space~$\Xc$ with the choice of the topology~${\mathpzc{t}}$ as presented in Theorem~\ref{the_gamma-intro}. In particular, the topology has to be sufficiently weak to assure the convergence of maximizing sequences and sufficiently strong to allow us to find the desired energy concentration.

The reason for the choice of such a space~$\Xc$ is very natural.
First of all, we are interested in the asymptotic behaviour of the sequence $\{\Fep(u_{\eps})\}$ for every sequence~$\{u_{\eps}\}$ such that $\|D_H u_{\eps}\|^{2}_{L^{2}(\Om)}\leq 1$. Such a constraint on the {\it horizontal Dirichlet energy} of~$u_\eps$ implies that, up to subsequences, there exists $\mu \in \mea$ and $u \in \Sc(\Om)$ such that~$\mu(\Omb)\leq 1$, $|D_H u_{\eps}|^{2}{\rm d}\xi\tows \mu$ in $\mea$ and $u_\eps \rightharpoonup u$ in $\Sc$.
Clearly, by the Sobolev embedding, we also get $u_{\eps} \rightharpoonup u$ in $L^{2^{\ast}}\!(\Om)$.
 \vspace{1mm}
By Fatou's Lemma, we can then deduce that~$\mu\geq|D_H u|^{2}{\rm d}\xi$, so that we can always decompose $\mu$ as follows,
\begin{equation}\label{mu}
\mu=|D_H u|^{2}{\rm d}\xi+\tilde{\mu}+\sum_{j=1}^{\infty}\mu_{j}\delta_{\xi_{j}},
\end{equation}
where $\mu_{j} \in [0,1]$ and $\{\xi_{j}\} \subseteq \overline{\Omega}$ are distinct points; the positive measure $\tilde{\mu}$ can be viewed as the {\it non-atomic part} of the measure $(\mu-|D_H u|^2{\rm d}\xi)$. \vspace{1mm}
In view of such a decomposition, the definition of the space~$\Xc$ given in Theorem~\ref{the_gamma-intro} is extremally natural. Also, the space~$\Xc$ is sequentially compact in the topology ${\mathpzc{t}}$. Indeed, if we take a sequence $\{(u_k, \mu_k)\}\subseteq \Xc$, then $\{u_k\}$ is bounded in $\Sc(\Om)$. Up to subsequences, $\mu_k\tows\mu$ in $\mea$ and $u_k \rightharpoonup u$ in $\Sc(\Om)$ (and in $L^{2^\ast}\!(\Om)$ by the Sobolev embedding) and the inequalities defining~$\Xc$ do still hold for $(u,\mu)$ by weak lower semicontinuity.
In view of the previous analysis, the space~$\Xc$ appears as a sort of completion of  the unit ball of~$\Sc(\Om)$ in the weak topology of the product $L^{2^\ast}\!(\Om)\!\times\mea$.

\vspace{2mm}
We now recall the definition of $\Gamma^+$-convergence adapted to our specific framework.
\begin{defn}
We say that the family $\{\Fep\} \ \Gamma^{+}$-converges to a functional $\Fc: \Xc \rightarrow [0,\infty)$ as $\eps \to 0$ if for every $(u,\mu) \in \Xc$ the following conditions hold:
\begin{itemize}
\item[(i)]{for every sequence $\{(u_{\eps},\mu_\eps)\} \subset \Xc$ such that $ u_{\eps}\rightharpoonup u$ in $L^{2^{\ast}}\!(\Om)$ and $\mu_\eps\tows \mu$ in $\mea$
\[
\Fc(u,\mu)\geq \limsup_{\eps \to 0}\Fep(u_{\eps},\mu_\eps);
\]
}
\item[(ii)]{there exists a sequence $\{(\bar{u}_{\eps}, \mbu)\} \subset \Xc$ such that $\bar{u}_{\eps} \tow u$ in $L^{2^{\ast}}\!(\Om)$, $\bar \mu_\eps \tows \mu$ in $\mea$ and
\[
\Fc(u,\mu)\leq \liminf_{\eps \to 0}\Fep(\bar{u}_{\eps}, \bar \mu_\eps).
\]
}
\end{itemize}
\end{defn}

\subsection{Proof of the $\Gamma^{+}$-limsup inequality}
The $\Gamma^{+}$-limsup inequality~(i) follows from the Concentration-compactness alternative  stated in Section~\ref{sec_preliminaries} via plain application of the H\"older inequality; we have the following
\begin{prop}\label{pro_limsup}
Let $\{\Fep\}$ be the family defined in~\eqref{def_fue1}, and let $\Fc$ be the functional defined in~\eqref{def_fu}. Then, for every $(u,\mu) \in \Xc$ and for every sequence~$\{(u_{\eps},\mu_{\eps})\} \subset \Xc$ such that $(u_{\eps},\mu_{\eps})\towt (u,\mu)$,
\[
\Fc(u,\mu) \geq \limsup_{\eps \to 0}\Fep(u_{\eps},\mu_{\eps}).
\]
\end{prop}
\vspace{0.5mm}
\begin{proof}
Let $\{(u_{\eps},\mu_{\eps})\}$ be a sequence in $\Xc$ such that $(u_{\eps},\mu_{\eps})\towt(u,\mu)$. This yields, as in~\eqref{mu},  that $\mu=|D_Hu|^{2}{\rm d}\xi+\tilde{\mu}+\sum_{j}\mu_{j}\delta_{\xi_{j}}$,  
{for some positive measure~$\tilde{\mu}\in {\mathcal{M}(\Omb)}$},  $\{\mu_j\}\subseteq(0,1]$ and $\{\xi_j\}\subseteq\Omb$.
Then, up to subsequences, there exists a measure~$\nu \in \mea$ such that $|u_{\eps}|^{2^{\ast}}\!{\rm d}\xi \tows \nu$,
and by the Concentration-compactness alternative in~Theorem~\ref{thm_cca} there exist nonnegative numbers $\{\nu_j\}$ such that (up to reordering the points $\{\xi_j\}$ and the~$\{\mu_j\}$)
\begin{equation}\label{eq_cca}
\dys \nu=|u|^{2^\ast}\!{\rm d}\xi+\sum_{j}\nu_{j}\delta_{\xi_{j}}  \ \ \text{and}\ \ \nu_{j}\leq S^{\ast}\mu_{j}^{\frac{2^{\ast}}{2}}.
\end{equation}

We now apply the H\"older inequality, which yields
\begin{eqnarray*}
\dys \Fep(u_{\eps},\mu_{\eps}) \ \equiv \ \int_{\Omega}|u_{\eps}|^{2^{\ast}-\eps}{\rm d}\xi \ \leq \ \left(\int_{\Omega}|u_{\eps}|^{2^{\ast}}{\rm d}\xi\right)^{\!\!\frac{2^{\ast}-\eps}{2^{\ast}}}\!|\Omega|^{\frac{\eps}{2^{\ast}}}.
\end{eqnarray*}
Thus, in view of the definition of $\nu$ and the decomposition~\eqref{eq_cca} in Theorem~\ref{thm_cca}, we have
\begin{eqnarray*}
 \dys \limsup_{\eps\to 0^{+}}\Fep(u_{\eps},\mu_{\eps})
& \leq &  \limsup_{\eps\to 0}\left(\int_{\Omega}|u_{\eps}|^{2^{\ast}}{\rm d}\xi\right)^{\!\!\frac{2^{\ast}-\eps}{2^{\ast}}}\!|\Omega|^{\frac{\eps}{2^{\ast}}} \leq \ \nu(\Omb) \\
& \leq &  \int_{\Omega}|u|^{2^{\ast}}{\rm d}\xi + S^{\ast}\sum_{i=1}^{\infty}\mu_{i}^{\frac{2^{\ast}}{2}} \  \leq \ \Fc(u,\mu),
\end{eqnarray*}
which gives the desired $\Gamma^+$-limsup inequality~(i).
\end{proof}

\vspace{2mm}
\subsection{Proof of the $\Gamma^{+}$-liminf inequality}
The proof of the $\Gamma^+$-liminf inequality~(ii) is extremely delicate, and it differs considerably from the Euclidean case (see, e.~g.,~\cite{AG03,Pal11}) where it has been proven via compactness and locality properties of the $\Gamma^+$-limit. Here, we are closer to the strategy presented in the fractional framework in~\cite{PPS15}, though in the various steps of our proof 
we need to deal with the Heisenberg framework, and we cannot make use of some fractional workarounds and well-established results for the fractional optimal maximizers.
\vspace{2mm}

For any fixed~$N\in\mathbb{N}$, the proof will be performed in the auxiliary space of configurations~${\Xc_N}\subseteq\Xc$ defined as follows,
{\begin{multline}\label{def_xn} 
\Xc_N :=\!\!\Big\{ \!(u,\mu)\!\in \Sc(\Om)\times \mea \,: 
\, \mu=|D_H u|^2{\rm d}\xi+\tilde{\mu} + \!\!\sum_{j=1}^N \mu_j\delta_{\xi_j}, \\ \mbox{with } \mu_{j} \in [0,1], \; \{\xi_{j}\} \subseteq \overline{\Omega},\; \tilde{\mu}\in\mea,\; \mu(\Omb)<1 \Big\}.
\end{multline}}
Since~$\Xc_N$ is ${\mathpzc{t}}$-sequentially dense in~$\Xc$ by a precise approximation, and $\Fc$ is continuous with respect to such an approximation -- see the computations below -- the $\Gamma^+$-liminf inequality will immediately follow by a standard diagonal argument.
\vspace{1mm}

Indeed, for any pair $(u,\mu)\in \Xc$, with $\mu$ decomposed as in \eqref{mu}, it suffices to consider the sequence~$\{(u_N, \mu_N)\}$ defined as follows,
\[
\dys u_N :=   a_N u \ \ \ \text{and} \ \ \
\dys \mu_N  :=   a^2_N|D_H u|^2{\rm d}\xi +  a^2_N\tilde{\mu} + a^2_N\sum_{j=1}^N \mu_j\delta_{\xi_j}\,,
\]
where $\{a_N\}\subset (0,1)$ is any increasing sequence such that $ a_N \to 1$ as $N\to \infty$.
\vspace{1mm}

By definition, the sequence $\{(u_N,\mu_N)\}$ belongs to~${\Xc_N}$, since, for any $N\in \N$, $u_N \in \Sc(\Om)$ and $\mu_N$ is a measure with a finite number of atoms such that $\dys \mu_N(\Omb) \, \leq \,  a^2_N\,\mu(\Omb) \, \leq \, a^2_N \, < 1$. Also, $(u_N, \,\mu_N) \towt (u,\mu)$ as $N\to \infty$, because $u_N \to u$ in $\Sc(\Om)$ (hence weakly in $L^{2^\ast}\!(\Om)$), and, for any $\phi \in C^0_0(\Omb)$,
\begin{eqnarray*}
\dys
 \int_{\Omb} \phi {\rm d}\mu_N & = & \int_{\Om} \phi a^2_N|D_H u|^2{\rm d}\xi + \int_{\Om}\phi a^2_N {\rm d}\tilde{\mu} + a^2_N \sum_{j=1}^N \mu_j\phi({\xi}_j) \\
& = & a^2_N\left(\int_{\Om} \phi |D_H u|^2{\rm d}\xi + \int_{\Om}\!\phi  {\rm d}\tilde{\mu}+ \sum_{j=1}^N \mu_j\phi(\xi_j)\right) \ \overset{N\to \infty}{\longrightarrow} \ \int_{\Om}\!\phi {\rm d}\mu\,.
\end{eqnarray*}
\vspace{1mm}

With such an approximation in mind, we also have
\begin{eqnarray*}
\dys
\Fc(u_N, \mu_N) & = & \int_\Om a^{2^\ast}_N |u|^{2^\ast}{\rm d}\xi + S^{\ast}\sum_{j=1}^N(a^2_N\mu_j)^{\frac{2^\ast}{2}} \\
& = & a^{2^\ast}_N \left(\int_\Om|u|^{2^\ast}{\rm d}\xi + S^{\ast}\sum_{j=1}^N \mu_j^{\frac{2^\ast}{2}}\right) \ \to \ \Fc(u,\mu), \ \ \text{as} \ n \to \infty.
\end{eqnarray*}

\vspace{1mm}

We are  in a position to perform the $\Gamma^+$-liminf inequality in the auxiliary space~$\Xc_N$; i.~\!e.,
\begin{prop}\label{prop_gammaliminf}
Let $\{\Fep\}$ be the family defined in~\eqref{def_fue1}, and let $\Fc$ be the functional defined in~\eqref{def_fu}. Then, for every $(u,\mu)$ in the auxiliary space~$\Xc_N$ defined in~\eqref{def_xn} there exists a sequence~$\{(\bar{u}_{\eps},\bar{\mu}_{\eps})\} \subset$ {$\Xc$} such that $(\bar{u}_{\eps},\bar{\mu}_{\eps})\towt (u,\mu)$ and
\begin{equation}\label{liminfN}
\Fc(u,\mu) \leq \liminf_{\eps \to 0}\Fep(\bar{u}_{\eps},\bar{\mu}_{\eps}).
\end{equation}
\end{prop}
The proof consists of the following steps:
\begin{itemize}
\item For any point~$\xi_j\in\Omb$, we construct a sequence~$\{u^{(j)}_\eps\}$ that concentrates energy at~$\xi_j$; see Lemma~\ref{prop_step1}.\vspace{0.5mm}
\item For any pair~$(0,\mu)$ such that $\mu=\sum_j\mu_j$ is purely atomic, we show how to glue the concentrating sequences of the previous step into an approximating sequence
\[
\{(\ue^{(\Sigma)}, |D_H\ue^{(\Sigma)}|^2{\rm d}\xi)\}_\eps.
\]
See Lemma~\ref{lemma_step2} below.\vspace{0.5mm}
\item {We present a   cut-off argument in order to finally construct the desired recovery sequence~$(\bar{u}_\eps, \bar{\mu}_\eps)$ for any pair~$(u,\mu)$ in the auxiliary space~$\Xc_N$.}
\end{itemize}
\begin{lemma}\label{prop_step1}
For any~$\xi_j\in\Omb$ there exists a sequence~$\{u^{(j)}_\eps\}\subset \Sc(\Om)$ such that
\begin{equation}\label{step1}
 { \Xc} \supseteq \big\{(u^{(j)}_\eps, |D_H\ue^{(j)}|^2{\rm d}\xi)\big\} \ {\mathpzc{t}}{\text{-converges to}}\ (0,\delta_{\xi_j}) \ \text{as } \eps \to 0.
\end{equation}
Moreover, the sequence~$\{u^{(j)}_\eps\}$ satisfies
\begin{equation}\label{step1_1}
\lim_{\eps\to 0}\int_\Om \snr{u^{(j)}_\eps}^{2^*-\eps} \, {\rm d}\xi = S^*,
\end{equation}
with~$S^*$ being the best Sobolev constant in~$\h$.
\end{lemma}
\begin{proof}
	We begin by assuming that~$\xi_j$ belongs to the interior of~$\Om$, and we modify the extremal functions~$U$ given by Theorem~\ref{thm_optimal}. Define
	\[
	w_\eps(\xi) := \eps^{-\frac{Q-2}{2}}U_{\eps,\xi_j}(\xi), \qquad \forall \xi \in \h.
	\]
	It follows by the very definition that, for any~$\rr >0$, the function~$w_\eps$ satisfies
	\begin{equation}\label{prop_step1 1}
		w_\eps \to 0 \quad \text{in} \ L^{2^*}(\h \setminus \overline{B_\rr(\xi_j)})
		\end{equation}
		and
			\begin{equation}\label{prop_step1 1bis}
			 \snr{D_H w_\eps}^2 \to 0 \quad \text{in} \ L^1(\h \setminus \overline{B_\rr(\xi_j)})\,
	\end{equation}
	whenever~$\eps \to 0$. Indeed, a change of variables yields
	\begin{eqnarray*}
		\int_{\h \setminus \overline{B_\rr(\xi_j)}} \snr{w_\eps}^{2^*} \, {\rm d}\xi &=&  \int_{\h \setminus \overline{B_\rr(\xi_j)}} \eps^{-Q}\big|U(\delta_{\frac{1}{\eps}}(\xi_j^{-1}\circ \xi))\big|^{2^*} \, {\rm d}\xi \\*
		&=& \int_{\h \setminus \overline{B_\rr(0)}} \eps^{-Q}\big|U\big(\delta_{\frac{1}{\eps}}(\eta)\big)\big|^{2^*}\, {\rm d}\eta \\*
		&=& \int_{\h \setminus \overline{B_{\frac{\rr}{\eps}}(0)}} \big|U(\eta')\big|^{2^*}\, {\rm d}\eta' \xrightarrow{\eps \rightarrow 0}0.
	\end{eqnarray*}
	In a similar way, the convergence~$\snr{D_H w_\eps}^2 \to 0$ in~$L^1(\h \setminus \overline{B_\rr(\xi_j)})$ follows by a change of variables and recalling the one~$\delta_\lambda$-homogeneity of the horizontal gradient.
	
	Fix~$\rr>0$ and take~$\varphi \in C^\infty_0(\h)$ such that~$\varphi \equiv 0$ on~$\h \setminus B_{2\rr}(\xi_j)$,~$\varphi \equiv 1$ on~$B_\rr(\xi_j)$ and~$0 \leq \varphi \leq 1$. We use the cut-off function~$\varphi$ to localize~$w_\eps$ in a smaller neighborhood of~$\xi_j$. In particular, we set	\[
	\tilde{u}_\eps^{(j)}(\xi):= \varphi(\xi)w_\eps(\xi)\,,
	\]
	and we claim that as~$\eps \rightarrow 0$ the following convergence results do hold,
	\begin{equation}\label{prop_step1 2}
		\tilde{u}_\eps^{(j)} \rightharpoonup 0 \quad \text{in } L^{2^*}(\Om)\,,
		\end{equation}
			\begin{equation}\label{prop_step1 2bis}
			 \|D_H \tilde{u}^{(j)}_\eps\|_{L^2(\h)} \to 1\,,
		\end{equation}
and		
	\begin{equation}\label{prop_step1 2ter}
		\int_{\Om}\snr{\tilde{u}_\eps^{(j)}}^{2^* -\eps} \, {\rm d}\xi \rightarrow S^*.
	\end{equation}
	The first convergence in~\eqref{prop_step1 2} plainly follows from~$\eqref{prop_step1 1}$. 
	
	In order to prove the claim in~$\eqref{prop_step1 2bis}$, it suffices to carefully estimate  the horizontal Dirichlet energy; we have
	\begin{eqnarray*}
		\int_{\h} \snr{D_H \tilde{u}^{(j)}_\eps}^2 \, {\rm d}\xi
		&=& \int_{B_{2\rr}(\xi_j) \setminus B_\rr(\xi_j)} \snr{D_H \tilde{u}^{(j)}_\eps}^2 \, {\rm d}\xi + \int_{B_\rr(\xi_j)} \snr{D_H w_\eps}^2 \, {\rm d}\xi \nonumber\\*
		&=& \int_{B_{2\rr}(\xi_j) \setminus B_\rr(\xi_j)} \snr{D_H \tilde{u}^{(j)}_\eps}^2 \, {\rm d}\xi + \int_{B_\frac{\rr}{\eps}(0)} \snr{D_H U(\xi)}^2 \, {\rm d}\xi\,,
	\end{eqnarray*}
	where the last identity plainly follows in view of the~$\delta_\lambda$-homogeneity of the horizontal gradient. Hence,~$\eqref{prop_step1 2bis}$ is satisfied once we proved that the first integral in the preceding identity  goes to~$0$ whenever~$\eps \to 0$. We have form~\eqref{prop_step1 1bis}
	\begin{eqnarray*}
		&& \int_{B_{2\rr}(\xi_j) \setminus B_\rr(\xi_j)} \snr{D_H \tilde{u}^{(j)}_\eps}^2 \, {\rm d}\xi  \\*[1ex]
		&&\quad \qquad \leq \ 2\| \varphi\|_{L^\infty}^2\int_{B_{2\rr}(\xi_j) \setminus B_\rr(\xi_j)} \snr{D_H w_\eps}^2 \, {\rm d}\xi + \, 2\|D_H \varphi\|_{L^\infty}^2\int_{B_{2\rr}(\xi_j) \setminus B_\rr(\xi_j)} \snr{w_\eps}^2 \, {\rm d}\xi \\*[1ex]
		&& \quad \qquad\leq \  c \int_{\h \setminus B_\rr(\xi_j)} \snr{D_H w_\eps}^2 \, {\rm d}\xi+c \eps^2 \int_{\h \setminus B_\frac{\rr}{\eps}(0)} \snr{U}^2 \, {\rm d}\xi \ \xrightarrow{\eps \to 0} 0.
	\end{eqnarray*}

    As for the claim in~$\eqref{prop_step1 2ter}$, it is convenient to split the integral as follows,
	\begin{eqnarray}\label{i12}
	\int_{\Om}\snr{\tilde{u}_\eps^{(j)}}^{2^* -\eps} \, {\rm d}\xi
	& =&  \int_{\Om \cap \{w_\eps <1\}}\snr{\tilde{u}_\eps^{(j)}}^{2^* -\eps} \, {\rm d}\xi + \int_{\Om \cap \{w_\eps \geq 1\}}\snr{\tilde{u}_\eps^{(j)}}^{2^* -\eps} \, {\rm d}\xi \notag \\*[1ex]
	& =: & {I}_{1,\eps} + {I}_{2,\eps}.
	\end{eqnarray}
	Firstly, note that~$\snr{\varphi(\xi)w_\eps(\xi)}^{2^* -\eps} <1$ on~$\Om \cap \{w_\eps <1\}$ uniformly with respect to~$\eps$, and~$\varphi(\xi)w_\eps(\xi) = {\rm o}(\eps^{Q-2})$
	as  $\eps\to 0$, since in $\Om \cap \{w_\eps <1\}$ it is away from the concentration point.
	This yields
	\begin{equation}\label{i1limit}
	{I}_{1,\eps}  \rightarrow 0 \ \text{as}\ \eps \to 0\,.
	\end{equation}

	Now,
	recall that for~$\eps $ sufficiently small,~$\Om \cap \{w_\eps \geq 1\} \subset B_\rr (\xi_j)$ 	and
	\begin{equation}\label{eq001}
	\frac{\varphi^{2^*-\eps}}{w_\eps^\eps} - 1 = \frac{1}{w_\eps^\eps} - 1\,.
	\end{equation}
	Also, on~$\Om \cap \{w_\eps \geq 1\}$ we have
	\begin{equation}\label{eq002}
	1 \leq w_\eps^\eps \leq (\max w_\eps)^\eps \leq c_0^\eps \eps^{-\eps\frac{Q-2}{2}} \xrightarrow{\eps \to 0}1\,.
	\end{equation}
We are finally ready to estimate the integral~${I}_{2,\eps}$ in~\eqref{i12}; we have
	\begin{eqnarray}\label{i2limit}
	{I}_{2,\eps}
	\ & \leq &
	 \left\| \frac{\varphi^{2^*-\eps}}{w_\eps^\eps} - 1 \right\|_{L^\infty(\{w_\eps \geq 1\})}\int_{\Om \cap \{w_\eps \geq 1\}}\snr{w_\eps}^{2^*} \, {\rm d}\xi \notag\\*[0.5ex]
	 &&  + \int_{\Om \cap \{w_\eps \geq 1\}}\snr{w_\eps}^{2^*} \, {\rm d}\xi \xrightarrow{\eps\to0} S^*,
	\end{eqnarray}
	where we used~\eqref{eq001}-\eqref{eq002} and also the definition of~$w_\eps$ in view of the optimality of the extremal functions~$U$ in~Theorem~\ref{thm_optimal}.
	\vspace{1mm}
	
	Passing to the limit as~$\rr \to 0$, by a diagonal argument,
	for any~$\eps>0$, we now set~$u^{(j)}_\eps :=  \tilde{u}^{(j)}_\eps /\|D_H \tilde{u}^{(j)}_\eps\|_{L^2(\h)}$. Clearly, it holds that the pair~$(u^{(j)}_\eps, \snr{D_Hu^{(j)}_\eps}^2 {\rm d}\xi)$ belongs to the space~${\Xc}$. Indeed,
	\[
	(u^{(j)}_\eps, \snr{D_Hu^{(j)}_\eps}^2 {\rm d}\xi) \in \Sc(\Om) \times \mea,
    \quad \text{and} \quad
	\int_{\overline{\Om}} \snr{D_Hu^{(j)}_\eps}^2 {\rm d}\xi 	  \leq 1.
	\]
	All in all, combining \eqref{prop_step1 2}-\eqref{prop_step1 2ter}
	and \eqref{i12} with~\eqref{i1limit} and~\eqref{i2limit},
the desired $\Gamma^+$-liminf result does follow in the case when~$\xi_j$ is an interior point.
	\vspace{1mm}
	
	The case when~$\xi_j$ belongs to~$\partial \Om$ can be recovered by considering an approximating sequence of interior points~$\xi_j^k$ converging to~$\xi_j$ and their corresponding optimal functions.
\end{proof}

We are now ready to construct a suitable recovery sequence for pairs~$(0,\mu)$ having a purely atomic measure component.
\begin{lemma}\label{lemma_step2}
For any finite set of distinct points $\{\xi_1, \xi_2, ..., \xi_N\}\subset \Omb$ and for any set of positive numbers~$\{\mu_1, \mu_2, ..., \mu_N\}\subseteq\mathbb{R}$ such that $\sum_j\mu_j<1$, there exists a sequence $\{u_\eps^{(\Sigma)}\}\subset \Sc(\Om)$ such that
{\begin{equation}\label{recovery_step2}
  \left\{\left(u^{(\Sigma)}_\eps, |D_H u_\eps^{(\Sigma)}|^2{\rm d}\xi\right)\right \} \subset { \Xc} \mbox{ and } \left(u^{(\Sigma)}_\eps, |D_H u_\eps^{(\Sigma)}|^2{\rm d}\xi\right)\towt \left(0, \sum_{j=1}^{N}\mu_j\delta_{\xi_j}\right)
\end{equation}}
as $\eps\to 0$.
Moreover, the sequence~$\{u^{(\Sigma)}_\eps\}$ satisfies
\begin{equation}\label{liminf_step2}
\lim_{\eps\to 0}\int_\Om \snr{u^{(\Sigma)}_\eps}^{2^*-\eps} \, {\rm d}\xi = S^*\sum_{j=1}^N \mu_j^{\frac{2^\ast}{2}}\,.
\end{equation}

\end{lemma}

\begin{proof} Put $B_j:=B_{r_j}(\xi_j)\cap\Omega$ for any $j=1,2,\dots,N$, for suitable radii $r_j$ and $r_i$ such that
${\rm dist}(B_j,B_i)>0$. By Lemma~\ref{prop_step1} there exists
a sequence~$\{u^{(j)}_\eps\}\subset \Sc(\Om)$ such that
$(u^{(j)}_\eps, |D_H\ue^{(j)}|^2{\rm d}\xi) \towt (0,\delta_{\xi_j}) \ \text{as } \eps \to 0.$
Moreover, from the proof of Lemma~\ref{prop_step1} we immediately deduce  that ${\rm spt \, \ue^{(j)}}\subset B_j$,  ${\rm dist}\big({\rm supp}\, \ue^{(j)}, \,\{\xi_j\}\big)\to 0$ as $\eps \to 0$, and
\[\int_{\Om} \snr{u^{(j)}_\eps}^{2^*-\eps} \,{\rm d}\xi\xrightarrow{\eps \to 0} S^*\quad \text{for }j=1,2,\dots,N.\]
Let us set
$u^{(\Sigma)}_\eps:=\sum_{j=1}^{N}\sqrt{\mu_j}u^{(j)}_\eps$.
A simple computation gives
\begin{eqnarray}\label{lemma_step2 1}
\int_{\h}\snr{D_Hu^{(\Sigma)}_\eps}^2 {\rm d}\xi &=&\sum_{j=1}^{N}{\mu_j}\int_{\h}\snr{D_Hu^{(j)}_\eps}^2 {\rm d}\xi +2\sum_{\substack{i,j=1 \notag\\*[0.5ex]
i<j}}^{N}\sqrt{\mu_i\mu_j}\int_{\h} D_Hu^{(i)}_\eps  D_Hu^{(j)}_\eps \, {\rm d}\xi \\
&=& \sum_{j=1}^{N}{\mu_j}\int_{\h}\snr{D_Hu^{(j)}_\eps}^2 {\rm d}\xi,
\end{eqnarray}
since
$
\int D_Hu^{(i)}_\eps  D_Hu^{(j)}_\eps \, {\rm d}\xi= 0 \text{ for } i\neq j,
$
being $\ue^{(i)}$ and $\ue^{(j)}$ with disjoint support whenever $i\neq j$.
Combining~\eqref{lemma_step2 1} with the convergence in~\eqref{step1}, we deduce that
\[
|D_Hu^{(j)}_\eps|^2{\rm d}\xi \tows \sum_{j=1}^{N}\mu_j\delta_{\xi_j} \text{ in } \mea.
\]
Finally, since $\sum_{j} \mu_j<1$, we deduce by~\eqref{lemma_step2 1} that $\int_{\h}\snr{D_Hu^{(\Sigma)}_\eps}^2 {\rm d}\xi\leq 1$ for $\eps$ small enough. Consequently, $\{(u^{(\Sigma)}_\eps, |D_H\ue^{(\Sigma)}|^2{\rm d}\xi)\}\subset  { \Xc}$, as required. This completes the proof of~\eqref{recovery_step2}. The last claim given in~\eqref{liminf_step2} is a plain consequence of~\eqref{step1_1}; we have
\[
\int_\Om \snr{u^{(\Sigma)}_\eps}^{2^*-\eps} \, {\rm d}\xi \vspace{-4mm}
\ =\ \sum_{j=1}^N \mu_j^\frac{2^*-\eps}{2}\int_\Om \snr{u^{(j)}_\eps}^{2^*-\eps} \, {\rm d}\xi \ \xrightarrow{\eps \to 0}\ \sum_{j=1}^N \mu_j^\frac{2^*}{2}S^*\,.
\]
\end{proof}

\vspace{2mm}
\begin{proof}[\bf Proof of~{Proposition~{\ref{prop_gammaliminf}}}]
Fix~$(u,\mu) \in \Xc_N$; i.~\!e.,~$u \in L^{2^*}(\Om)$ and
	\[
	\mu = \snr{D_H u}^2 {\rm d}\xi + \tilde{\mu} +\sum_{j=1}^N \mu_j \delta_{\xi_j} \in \mea,
	\]
	such that~$\mu (\overline{\Om}) <1$, and consider radius~$\rr$ converging to~$0$ such that the balls~$B_{2\rr}(\xi_j) \cap B_{2\rr}(\xi_i) = \emptyset$, for any~$i,j \in \{1,\dots,N\}$ with~$i\neq j$. Denote with
	\[
	B_{\rr,j}:= B_{\rr}(\xi_j) \cap \overline{\Om} \quad \text{and} \quad
	\varphi_\rr(\xi) = 1 - \sum_{j=1}^N \psi \big(\delta_{\frac{1}{\rr}}(\tau_{\xi_j}^{-1}(\xi)) \big),
	\]
	where~$\psi \in C^\infty_0(B_2)$,~$\psi \equiv 1$  on~$\overline{B_1}$,~$ 0 \leq \psi \leq 1$, with~$B_2 \equiv B_2(0)$. Note that~$\varphi_\rr \equiv 0$ in~$\overline{B_{\rr,j}}$,  for~$j=1,\dots,N$, and~$\varphi_\rr\equiv 1$ in~$\overline{\Om} \smallsetminus \bigcup_{j=1}^N B_{2\rr,j}$.
	
	Now denote with~$u^{(\Sigma)}_\eps$ the sequence given by Lemma~\ref{lemma_step2} and define
	\[
	\bar{u}_\eps:= u\varphi_\rr + u^{(\Sigma)}_\eps, \qquad \bar{\mu}_{\eps}:= \tilde{\mu} + \snr{D_H(u\varphi_\rr +u^{(\Sigma)}_\eps)}^2{\rm d}\xi.
	\]
	Taking first the limit as $\eps \to 0$ and then for $\rr \to 0$, by a diagonal argument, we prove that the above sequence is a recovery sequence for $(u,\mu)$.
	
	First we prove that $(\bar{u}_\eps,\bar{\mu}_\eps)$ for $\eps$ and $\rr$ sufficiently small belongs to the space  {$\Xc$}. By construction we are only left to prove that
	\begin{equation}\label{prop_gammaliminf_1}
		\bar{\mu}_\eps (\overline{\Om}) \,{\leq}\, 1.
	\end{equation}
	Indeed, for any $\eps,\rr >0$ we have that
	\begin{eqnarray*}
		\bar{\mu}_\eps (\overline{\Om}) &=&\tilde{\mu}(\overline{\Om}) + \int_{\Om}\snr{D_H (u\varphi_\rr+u^{(\Sigma)}_\eps)}^2 {\rm d}\xi\\
		&=& \tilde{\mu}(\overline{\Om}) + \int_{\Om}\snr{D_H (u\varphi_\rr)}^2 {\rm d}\xi  + \int_{\Om}\snr{D_H u^{(\Sigma)}_\eps}^2 {\rm d}\xi,
	\end{eqnarray*}
	where we have used the fact that $\int_\Om D_H u\varphi_\rr\cdot D_H u^{(\Sigma)}_\eps\, {\rm d}\xi=0$, for $\eps$ sufficiently small. Moreover,  by Lemma~\ref{lemma_step2} we have that
	\begin{equation}\label{prop_gammaliminf 4}
			\lim_{\eps \to 0} \int_{\Om} \snr{D_H u^{(\Sigma)}_\eps}^2 \, {\rm d}\xi = \sum_{j=1}^N \mu_j.
	\end{equation}
	On the other hand, we have that
	\begin{eqnarray}\label{prop_gammaliminf_3}
		\int_{\Om}\snr{D_H( u\varphi_\rr)}^2 \, {\rm d}\xi &=& \int_{\Om\setminus \bigcup_j B_{2\rr,j}}\snr{D_H u}^2 \, {\rm d}\xi
		 + \sum_{j=1}^N\int_{ B_{2\rr, j}\smallsetminus B_{\rr,j}}\snr{D_H (u \varphi_\rr)}^2 \, {\rm d}\xi \nonumber\\
		& \rightarrow& \int_{\overline{\Om}}\snr{D_H u}^2 \, {\rm d}\xi \ \text{as} \ \rr \to 0\,,
	\end{eqnarray}
since for any~$j \in \{1,\dots,N\}$ the integral~$\int_{ B_{2\rr, j}\smallsetminus B_{\rr,j}}\snr{D_H (u \varphi_\rr)}^2 \, {\rm d}\xi$ tends to~$0$ with respect to~$\rr$.
	\\ Indeed, since~$\psi \in C^\infty_0(B_2)$, we have
    \begin{eqnarray*}
&&		\int_{ B_{2\rr, j}\smallsetminus B_{\rr,j}}\snr{D_H (u \varphi_\rr)}^2 \, {\rm d}\xi
		\\
		&&\qquad \qquad \qquad \quad \le
		2\|D_H u\|^2_{L^2(B_{2\rr, j}\smallsetminus B_{\rr,j})} + 2\rr^\frac{Q-2}{2} \|u\|_{L^2(\Om)}\|D_H \psi\|_{L^2( B_{2}\smallsetminus B_{1})}\,,
\end{eqnarray*}
	which goes to~$0$ as~$\rr \to 0$. It is now sufficient to combine~\eqref{prop_gammaliminf 4} with~\eqref{prop_gammaliminf_3} to get that~\eqref{prop_gammaliminf_1} holds true.
	\vspace{1mm}
	
	We show now that~$\{(\bar{u}_\eps,\bar{\mu}_\eps)\}$ ${\mathpzc{t}}$-converges to $(u,\mu)$; i.~\!e.,
	\begin{equation}\label{prop_gammaliminf_2}
		\bar{u}_\eps \rightharpoonup u \quad \text{in}~L^{2^*}(\Om) \quad \text{and} \quad \bar{\mu}_\eps \tows \mu \quad \text{in}~\mea.
	\end{equation}
	Note that~$\eqref{prop_gammaliminf_2}_1$ follows from the fact that~$u^{(\Sigma)}_\eps \rightharpoonup 0$ in~$L^{2^*}(\Om)$, as shown in Lemma~\ref{lemma_step2}, and~$u\varphi_\rr$ converges (strongly) to~$u$ as~$\rr \to 0$. Indeed,
	on~$\Om$ we have that~$\snr{u\varphi_\rr-u}^{2^*}=\snr{1-\varphi_\rr}^{2^*}\snr{u}^{2^*} \leq \snr{u}^{2^*}$, thus by Lebesgue's Dominated Convergence Theorem we obtain that~$u\varphi_\rr \to u$ in~$L^{2^*}(\Om)$.
	
	The second convergence in~$\eqref{prop_gammaliminf_2}_2$ follows also from Lemma~\ref{lemma_step2}. Indeed, once fixed~$\phi \in C^0_0(\overline{\Om})$, we get that
	\begin{eqnarray*}
		\lim_{\rr \to 0} \lim_{\eps \to 0} \int_{\overline{\Om}}\phi \, {\rm d}\bar{\mu}_\eps &=& \int_{\overline{\Om}}\phi \, {\rm d}\tilde{\mu} + \lim_{\rr \to 0} \lim_{\eps \to 0} \int_{\Om}\phi \, \snr{D_H(u\varphi_\rr+u^{(\Sigma)}_\eps)}^2{\rm d}\xi\\[1ex]
		&=&  \int_{\overline{\Om}}\phi \, {\rm d}\tilde{\mu} + \lim_{\rr \to 0} \int_{\Om}\phi \, \snr{D_H(u\varphi_\rr)}^2{\rm d}\xi\\
		&& +\lim_{\eps \to 0} \int_{\Om}\phi \, \snr{D_Hu^{(\Sigma)}_\eps}^2{\rm d}\xi\\[1ex]
		&=&  \int_{\overline{\Om}}\phi \, {\rm d}\tilde{\mu} +  \int_{\Om}\phi \, \snr{D_Hu}^2{\rm d}\xi+\sum_{j=1}^N\int_{\overline{\Om}}\phi \mu_j \, {\rm d}\delta_{\xi_j}\\[1ex]
		&=& \int_{\overline{\Om}} \phi \, {\rm d}\mu,
	\end{eqnarray*}
where we also used~\eqref{prop_gammaliminf 4} and~\eqref{prop_gammaliminf_3}.

\vspace{1mm}

It remains to prove the liminf inequality in~\eqref{liminfN}. Firstly, note that the integral in the definition of~$\Fep(\bar{u}_\eps,\bar{\mu}_\eps)$ can be splitted as follows
\[
\Fep(\bar{u}_\eps,\bar{\mu}_\eps) = \int_\Om \snr{u\varphi_\rr}^{2^*-\eps} \, {\rm d}\xi + \int_\Om \snr{u^{(\Sigma)}_\eps}^{2^*-\eps} \, {\rm d}\xi\,,
\]
Moreover, as in the proof of~$\eqref{prop_gammaliminf_2}$, we can deduce that
\[
\lim_{\rr \to 0}\lim_{\eps\to 0}\int_\Om \snr{u\varphi_\rr}^{2^* -\eps} \, {\rm d}\xi = \int_\Om \snr{u}^{2^*} \, {\rm d}\xi.
\]
By combining all the previous results, up to a diagonal argument, we finally get
\[
\liminf_{\eps\to 0}\Fep(\bar{u}_\eps,\bar{\mu}_\eps)= \int_\Om \snr{u}^{2^*} \, {\rm d}\xi + S^* \sum_{j=1}^N \mu_j^\frac{2^*}{2} \equiv \Fc(u,\mu)\,
\]
as desired.
\end{proof}
\vspace{2mm}

\vspace{2mm}
\section{Proof of the concentration result}
\label{sec_concentration}
In this section we will prove~Theorem~\ref{cor_concentration} showing that, due to the $\Gamma^+$-convergence result,  the   maximizers~$\{\ue\}$ for the variational problem (\ref{problema}) concentrate energy at one point $\xi_{\rm o} \in \Omb$ when $\eps$ goes to zero.

\vspace{2mm}

Firstly, notice that, since the embeddings  $\Sc(\Om) \hookrightarrow L^{2^{\ast}-\eps}(\Om)$ are compact,
the functionals $\Fep$ as extended to $\Xc$ by~\eqref{def_fue1} are continuous.
As a consequence, we will have that the $\Gamma^{+}$-convergence of functionals in this space implies the convergence of maximizers~$\{(\ue,\mue)\}$ of $\Fep$ to the maxima of $\Fc$.  This will be the final step in the alternative proof of the concentration result presented at the end this section.
The first step is in Lemma~\ref{pro_argmax} below where we show that there are no further maximizers in the space~$X$ other then the pairs of the form $(\ue,|D_Hu_\eps|^2{\rm d}\xi)$.
\begin{lemma}\label{pro_argmax}
	For any $\eps>0$, let $(\bar{u}_\eps, \bar{\mu}_\eps)\in \Xc$ be such that
	\[
	\dys \sup_{(u,\mu)\in \Xc} \Fep(u,\mu) = \Fep(\bar{u}_\eps, \bar{\mu}_\eps).
	\]
	Then
	$
	\bar{\mu}_\eps=|D_H\bar{u}_\eps|^2{\rm d}\xi.
	$
\end{lemma}
\vspace{1mm}
\begin{proof}
	The present proof is a plain generalization of the one presented in the {\it fractional} (Euclidean) framework in~\cite{PPS15}, since it is based on the intrinsic scaling properties of the involved functional, and the underlying geometry does not interfere.
	
	Firstly, notice that the supremum is attained at some $(\bar{u}_{\eps},\mbu)$ because $\Xc$ is sequentially compact and $\Fep$ is sequentially continuous (due to the compact embedding~$\Sc(\Om) \hookrightarrow L^{2^\ast-\eps}(\Om)$).
	Clearly, it is not restrictive to assume that $\dys \bar{\mu}_\eps(\Omb)=1$.
	Indeed, if one has $\bar{\mu}_\eps(\Omb)<1$, then it suffices to consider the pair $(\bar{u}_\eps, \bar{\mu}_\eps/\bar{\mu}_\eps(\Omb))$ which still belongs to the space $\Xc$, it satisfies $(\bar{\mu}_\eps/\bar{\mu}_\eps(\Omb))(\Omb)=1$, and
	\[
	\Fep(\bar{u}_\eps, \bar{\mu}_\eps/\bar{\mu}_\eps(\Omb))=\Fep(\bar{u}_\eps, \bar{\mu}_\eps)=\max_{(u,\mu)\in \Xc} \Fep(u,\mu).
	\]
	Since $\bar{u}_\eps\neq0$, by the very definition of the space~$\Xc$, it follows $0<\|D_H\bar{u}_\eps\|_{L^2}\leq1$.
	Thus, by setting
	\begin{equation}\label{eq_recall}
		b\, =\, \dys b(\eps):=\, \frac{1}{\|\bar{u}_\eps\|^2_{\Sc(\Om)}}\, \geq \, 1,
	\end{equation}
	one can consider a new  pair $(\tilde{u}_\eps,\tilde{\mu}_\eps)$ given as follows,
	\[
	\dys \tilde{u}_\eps:={b^{\frac{1}{2}}}\bar{u}_\eps \ \ \text{and} \ \ \tilde{\mu}_\eps:=b|D_H\bar{u}_\eps|^2{\rm d}\xi.
	\]
	Notice that $(\tilde{u}_\eps,\tilde{\mu}_\eps)$ belongs to the space~$\Xc$ and it satisfies
	\begin{eqnarray}\label{eq_maxima}
		\Fep(\tilde{u}_\eps,\tilde{\mu}_\eps) \! & = & \! b^{\frac{{2^{\ast}-\eps}}{2}}\Fep(\bar{u}_\eps,\bar{\mu}_\eps) \, = \, b^{\frac{{2^{\ast}-\eps}}{2}}\max_{(u,\mu)\in \Xc} \Fep(u,\mu).
	\end{eqnarray}
	So that \eqref{eq_recall} and \eqref{eq_maxima} yield $b=1$,
	$(\bar{u}_\eps, |D_H\bar{u}_\eps|^2{\rm d}\xi)$ is a maximizer, and $\|\bar{u}_\eps\|_{\Sc(\Om)}=1$.
	Since $1=\int |D_H\bar{u}_\eps|^2{\rm d}\xi \leq \mbu(\Omb)=1$, one has $\mbu=|D_H\bar{u}_\eps|^2{\rm d}\xi$.
\end{proof}

The second step consists of Lemma~\ref{lem_sstar} below where we prove an optimal upper bound for the limit functional~$\Fc$ on the space $\Xc$. This will be the keypoint of the proof of the concentration result.
\begin{lemma}\label{lem_sstar}
	Let $\Fc$ be the functional defined in~\eqref{def_fu}. Then, for every $(u,\mu) \in \Xc$, we have
	\begin{equation}\label{eq_soboc}
		\Fc(u,\mu) \leq S^{\ast},
	\end{equation}
	and the equality holds if and only if $(u,\mu)=(0,\delta_{\xi_{\rm o}})$ for some $\xi_{\rm o} \in \Omb$.
\end{lemma}
\vspace{1mm}
\begin{proof}
	We adapt the argument in the proof of~\cite[Lemma~3.6]{AG03} for the Euclidean case in~$H^1_0(\Om)$, which essentially relies on the well-known ``convexity trick''.
	\vspace{1mm}
	
	For every $(u,\mu)\in \Xc$, the Sobolev inequality yields
	\[
	\Fc(u,\mu)  \equiv  \int_{\Om}|u|^{2^{\ast}}{\rm d}\xi+S^{\ast}\sum_{j=1}^{\infty}\mu_{j}^{\frac{2^{\ast}}{2}} \leq  S^{\ast}\left(\int_{\Om}|D_H u|^{2}{\rm d}\xi\right)^{\!\!\frac{2^{\ast}}{2}}+S^{\ast}\sum_{j=1}^{\infty}\mu_{j}^{\frac{2^{\ast}}{2}}.
	\]
	Now, by the convexity of the function $t\mapsto t^{\frac{2^{\ast}}{2}}\equiv t^{1+\frac{1}{n}}$, for every fixed~$n\geq 1$, it follows
	\begin{eqnarray}\label{numero}
		\Fc(u,\mu)  & \leq & S^{\ast}\left(\int_{\Om}|D_H u|^{2}{\rm d}\xi\right)^{\!\!\frac{2^{\ast}}{2}}+S^{\ast}\sum_{j=1}^{\infty}\mu_{j}^{\frac{2^{\ast}}{2}} \nonumber \\*
		& \leq &  S^{\ast}\left(\int_{\Om}|D_H u|^{2}{\rm d}\xi+\sum_{j=1}^{\infty}\mu_{j}\right)^{\!\!\frac{2^{\ast}}{2}}  \leq \, S^{\ast}\!\left(\mu(\Omb)\right)^{\!\frac{2^\ast}{2}} \, \leq \, S^{\ast}\,,
	\end{eqnarray}
	which proves~\eqref{eq_soboc},
	and the equality clearly holds if $(u,\mu)=(0,\delta_{\xi_{\rm o}}),$ for some $\xi_{\rm o} \in \Omb$. Indeed,
	let us assume that the equality in~\eqref{eq_soboc} holds for some pair~$(u,\mu) \in \Xc$. Then, each inequality in~\eqref{numero} is in fact an equality. This plainly implies that $\tilde{\mu}=0$. If~$u\neq 0$ then by convexity we also deduce that $\mu_j=0$ for every $j$. In turn, this fact yields $\mu=|D_H u|^2{\rm d}\xi$ and $u\in \Sc(\Om)$ is optimal in the Sobolev inequality, which contradicts Theorem~\ref{thm_optimal}. Thus, $u=0$, the equation in~\eqref{numero} and the strict convexity implies that $\mu=\delta_{\xi_{\rm o}}$ for some $\xi_{\rm o}\in \Omb$ as desired.
\end{proof}
\vspace{1mm}

We are finally in the  position to prove the expected concentration result.
\begin{proof}[\bf Proof of Theorem~{\rm\ref{cor_concentration}}]
	As mentioned at the beginning of the present section, by Theorem \ref{the_gamma-intro} and standard  $\gamp$-convergence properties, it follows that every sequence of maximizers of $\Fep$, which is in fact in the form  $\{(\ue, |D_H \ue|^2{\rm d}\xi)\}$ in view of Lemma~\ref{pro_argmax}, must converge (up to subsequences) to a pair $(u,\mu)\in \Xc$ which is a maximum for $\Fc$; that is,
	\[
	\dys (\ue,|D_H \ue|^2{\rm d}\xi)\towt(u,\mu), \ \ \ \text{with} \ \Fc(u,\mu)=\max_{(\bar{u},\bar{\mu}) \in X}\Fc(\bar{u},\bar{\mu}).
	\]
	Thanks to Lemma~\ref{lem_sstar}, we know that $\dys \Fc(u,\mu)\leq S^{\ast}$  for every $(u,\mu)\in \Xc$
	and that the equality is achieved if and only if $(u,\mu)=(0,\delta_{\xi_{\rm o}})$ for some $\xi_{\rm o}\in \Omb$. Hence, it follows that
	$\dys
	(\ue, |D_H \ue|^2{\rm d}\xi)\towt(0,\delta_{\xi_{\rm o}})$, which is the desired concentration property for the energy density.
\end{proof}

\vspace{2mm}
\section{Struwe's Global Compactness  via Profile Decomposition}\label{sec_struwe}
This section is devoted to the proof of Theorem \ref{thm_glob_comp}.
Before going straight into the proof a  remark is needed. Given a bounded domain with smooth boundary~$\Omega \subseteq \h$ we extend all functions of~$C^\infty_0(\Om)$ to the whole~$\h$ putting them equal zero outside~$\Om$. Then we can regard them as functions defined on the whole~$\h$. Now, define the  Folland-Stein Sobolev homogeneous space~$\Sc(\Om)$ as the completion of~$C^\infty_0(\Om)$  with respect to the norm~$\| D_H \cdot \|_{L^2(\h)}$.

Lastly, in the proof of Theorem~\ref{thm_glob_comp} we will make use of a fine asymptotic characterization result proved in~\cite{Ben08}, which presents to the sub-Riemannian setting of the Heisenberg group the Profile Decomposition firstly proven  by G\'erard in~\cite{Ger98}. See also the  recent alternative proof in~\cite{PP14} based upon the results in the very relevant book~\cite{TF07} where an abstract (even more general) approach in Hilbert spaces can be found; see in particular Section~9.9 there for related result in Carnot groups. We refer the reader also to the recent work~\cite{Tin20}.

\begin{theorem}[Theorem~1.1 in \cite{Ben08}]\label{thm_prof}
	Let~$\{u_k\}$ be a bounded sequence in~$ \Sc(\h)$. Then, for any~$j\in\mathbb{N}$ there exist a sequence of numbers~$\{\lambdakj\}\subset(0,\infty)$, a sequence of points~$\{\xikj\}\subset\h$ and a function~$\psi^{(j)}\in \Sc(\h)$ such that as~$k\to\infty$
	\begin{align}\label{strange}
		\begin{cases}
		 \left|\log{\frac{\lambda_k^{(i)}}{\lambdakj}} \right|\to\infty, & \text{if }\lambda_k^{(i)}\neq \lambdakj  \\\\
			\big|\delta_{1/\lambdakj}(\xikj^{-1}\circ \xi_k^{(i)})\big|_{\h}\to\infty &  \text{if }\lambda_k^{(i)}=\lambdakj,
		\end{cases}
	\end{align}
	and such that, for any integer~$\ell\in\mathbb{N}$,
	\begin{align}
		&u_k(\cdot)=\sum_{j=1}^{\ell}\lambdakj^{\frac{-Q}{2^*}}\psi^{(j)}(\delta_{1/\lambdakj}(\tau_{\xikj}^{-1}(\cdot)))+r^{(\ell)}_k(\cdot);\label{prof1}\\
		&\|u_k\|_{\Sc}^2=\sum_{j=1}^{\ell}\|\psi^{(j)}\|_{\Sc}^2+\|r^{(\ell)}_k\|_{\Sc(\h)}^2+o(1)\quad\text{as }k\to\infty;\label{prof2}\\
		&\|r^{(\ell)}_k\|_{L^{2^*}}\to 0\quad\text{as }k\to\infty \label{prof3}.
	\end{align}
\end{theorem}

\begin{proof}[\bf Proof of~{Theorem~{\ref{thm_glob_comp}}}]
We divide the proof into several steps.

\vspace{2mm}
\noindent{\bf Step 1. The sequence~$\{u_k\}$ is bounded in~$\Sc(\Om)$.}  
Since~$\{u_k\}$ is a Palais-Smale sequence for~$\El$, we have
\begin{eqnarray*}
\left(\frac12-\frac{1}{2^*}\right)\int_{\Om}|u_k|^{2^*}	\,{\rm d}\xi & =& \El(u_k)-\frac12 \langle {\rm d}\El(u_k),u_k\rangle\\*[1ex]
 &\leq& c +c|\langle{\rm d}\El(u_k), u_k\rangle| \\*[1.2ex]
 &\leq& c +c\eps\|u_k\|_{S^1_0}^2 + c_{\eps}o(1)\,. 
\end{eqnarray*}
Now, since~$\Om$ is bounded, we have as~$k\to\infty$
\begin{eqnarray*}\label{thm_glob_comp_1}
  \int_{\Om}|u_k|^{2}	\,{\rm d}\xi &\leq& |\Om|^{2/Q}\left(\int_{\Om}|u_k|^{2^{*}}\,{\rm d}\xi\right)^{2/2^{*}}\quad\text{as }k\to\infty\nonumber \\*[1ex] 
  &\leq& c+c_{\eps}o(1)+c\eps\|u_k\|_{S^1_0}^2,
\end{eqnarray*}
where we also used H\"older's Inequality and Young's Inequality.

\vspace{2mm} 
Hence, by combining the estimates above, we have that
\begin{eqnarray*}
\|u_k\|_{\Sc}^2&=&2\El(u_k)+\lambda \int_{\Om}|u_k|^{2}	\,{\rm d}\xi+\frac{2}{2^*}\int_{\Om}|u_k|^{2^*}	\,{\rm d}\xi\\*[1ex] 
&\leq& c +c\eps\|u_k\|_{S^1_0}^2 + c_{\eps}o(1) \quad\text{as }k\to\infty\,.
\end{eqnarray*}
So that, up to choosing~$\eps>0$ sufficiently small and reabsorbing the $\Sc$-norm of $u_k$ in the left hand-side, we have
\[
\|u_k\|_{\Sc}^2 \leq c + o(1) \quad\text{as }k\to\infty\,,
\]
which proves the claim.

\vspace{2mm}
\noindent{\bf Step 2. The weak limit~$\uz$ solves~\eqref{plambda}.} From Step 1., up to a subsequences not relabelled,  there exists
$\uz \in \Sc(\Omega)$ such that~$u_k\rightharpoonup\uz$ in~$\Sc(\Omega)$ as~$k\to\infty$. Hence, by the compact embedding theorem (see Theorem 3.2 of~\cite{GL92}), as ~$k\to\infty$
\begin{equation}\label{thm_glob_comp_conv}
\begin{aligned}
&u_k\rightharpoonup\uz\quad \quad\text{ weakly in } L^{2^*}(\Om),\\
&|u_k|^{2^*-2}u_k\to|\uz|^{2^*-2}\uz \quad \text{ strongly in }L^1(\Om),\\
&|u_k|^{2^*-2}u_k \rightharpoonup|\uz|^{2^*-2}\uz \quad\text{ weakly in } L^{(2^*)'}(\Om),\\
&u_k\to\uz\quad \text{ strongly in } L^2_{\text{\scriptsize{\rm loc}}}(\h).
\end{aligned}
\end{equation}
Thus, the weak continuity of the functional~${\rm d}\El$ yields that~$\uz$ is a solution to~\eqref{plambda}. If the convergence of $u_k$ to $\uz$ is strong we are done.

\vspace{2mm}
\noindent{\bf Step 3. If the convergence is not strong, then~$\{u_k\}$ contains further profiles.} Assume that~$u_k$ does not converge strongly to~$\uz$ in~$\Sc(\Omega)$. Then, we can apply Theorem~\ref{thm_prof} to the sequence~$\{u_k\}$, obtaining a profile decomposition for a renumbered
subsequence, that is for any integer~$\ell\in\mathbb{N}$ we can write
\begin{align*}
&u_k(\cdot)=\sum_{j=1}^{\ell}\lambdakj^{\frac{-Q}{2^*}}\psi^{(j)}(\delta_{1/\lambdakj}(\tau_{\xikj}^{-1}(\cdot)))+r^{(\ell)}_k(\cdot).
\end{align*}

Let~$\Irm=\{j\in\mathbb{N}\,|\, \psi^{(j)}\neq \uz\}$ and set~$\uj=\psi^{(j)}$ for~$j\in\Irm$.
Note that
\begin{itemize}
  \item By the definition of profiles, for any~$j\in\Irm$ there exist a sequence of numbers~$\{\lambdakj\}\subset(0,\infty)$, a sequence of points~$\{\xikj\}\subset\h$ such that
      \begin{equation}\label{conv_prof}
      u_k(\tau_{\xikj}(\delta_{\lambdakj}(\cdot)))\rightharpoonup\uj(\cdot)\quad\text{in }\Sc(\h).
      \end{equation}
  \item The limit function~$\uz$ can be regarded as the profile associated to the trivial dilations and translations.
\end{itemize}
Hence, if we show that ~$\Irm\neq \emptyset$, we are done proving that~$\{u_k\}$ contains further profiles.
Therefore, assume by contradiction that~$\Irm=\emptyset$ and so~$\{u_k\}$ does not contain profiles different from~$\uz$. Then, the sum in~\eqref{prof1} reduces to the single term~$\uz$ and~\eqref{prof1} gives
\[
u_k-\uz= \uz +r_k^{(0)} -\uz= r_k^{(0)},
\]
which in turn implies by~\eqref{prof3}
\[
\|u_k-\uz\|_{2^*}= \|r_k^{(0)}\|_{2^*}\to 0\quad\text{as }k\to\infty.
\]
Hence,~$u_k\to\uz$ strongly in~$L^{2^*}(\Om)$. On the other hand, as~$k\to\infty$
\begin{align*}
\|u_k-\uz\|^{2}_{\Sc(\Omega)}&=  ( u_k,u_k-\uz) - ( \uz,u_k-\uz) \\
&=\langle {\rm d}\El (u_k),u_k-\uz\rangle - \langle {\rm d}\El(\uz) ,u_k-\uz\rangle +\lambda \|u_k-\uz\|_{2}^2\\
&\quad+ \int_{\Om}\left(|u_k|^{2^*-2}u_k- |\uz|^{2^*-2}\uz\right)\left(u_k-\uz\right)	\,{\rm d}\xi=o(1),
\end{align*}
by~\eqref{PS2} and the convergence properties of $u_k$ in \eqref{thm_glob_comp_conv}. This gives the desired contradiction and so~$\Irm\neq \emptyset$; i.~\!e.,~$\{u_k\}$ contains further profiles.

\vspace{2mm}
\noindent{\bf Step 4. The scaling parameters~$\{\lambdakj\}$ from Theorem~\ref{thm_prof} satisfy~$\lambdakj\to 0$ and ${\rm dist}(\xikj,\Om)={\rm O}(\lambdakj)$ as~$k\to\infty$ for any~$j\in\mathbb{N}$.}
Arguing again by contradiction we first assume that, up to a subsequences,~$\lambdakj\to \infty$ as~$k\to\infty$.
Note that the convergence in~\eqref{conv_prof} is actually strong in~$L^2_{\text{\scriptsize{\rm loc}}}(\h)$ by Rellich's Theorem. Then, since~$\uj\neq 0$ we get by Fatou's lemma
\begin{align*}
0<\int_{\h}|\uj|^2	\,{\rm d}\xi &\leq \liminf_{k\to\infty}\lambdakj^{ Q-2 } \int_{\h} \left|u_k(\tau_{\xikj}(\delta_{\lambdakj}(\xi)))\right|^2	\,{\rm d}\xi\\
&=
\liminf_{k\to\infty}\lambdakj^{-2 } \int_{\Om}  |u_k |^2	\,{\rm d}\xi=0,
\end{align*}
which together with the fact that~$\lambdakj\to \infty$ and~\eqref{thm_glob_comp_1} gives a contradiction.

Therefore, there exists~$c^{(j)}>0$ such that~$\lambdakj\leq c^{(j)}$ for all~$k$. Now assume that, up to a subsequences,
$\lambdakj \geq c_j > 0$ for each~$k$. Then, since the profiles $\uj$ and $\uz$, for $j \in \Irm$, must be attained along different sequence of parameters and points satisfying conditions \eqref{strange}, again up to subsequences,~$\snr{\xikj}_{\h}\to\infty$.

Hence,~$ u_k(\tau_{\xikj}(\delta_{\lambdakj}(\cdot)))\rightharpoonup0$, since for any fixed~$k$, each function
is supported in~$\delta_{1/\lambdakj} (\tau_{\xikj^{-1}}(\Om))$, and so for~$k$ sufficiently large, they are zero on each compact set in~$\h$, $\Om$ being bounded and~$\lambdakj \geq c_j > 0$.
Since~$\uj\neq0$ we obtain a contradiction and therefore~$\lambdakj \to 0$ as~$k\to\infty$.

Let us now deal with the second part of the statement. Assume that
\[
\frac{{\rm dist}(\xikj,\Om)}{\lambdakj}\to\infty \qquad \text{as}~k\to\infty.
\]
This yields, possibly along  subsequences,~${\rm dist}({0},\delta_{1/\lambdakj} (\tau_{\xikj^{-1}}(\Om))) \to\infty$ {(where we denoted with~${0}$ the identity element in~$\h$)} and still, by arguing as before, we obtain that~$u_k(\tau_{\xikj}(\delta_{\lambdakj}(\cdot)))\rightharpoonup0$.
This is impossible and so~${\rm dist}(\xikj,\Om)={\rm O}(\lambdakj)$ as~$k\to\infty$.

\vspace{2mm}
\noindent{\bf Step 5. The profiles~$\uj$ solve~\eqref{pzero} either in a half-space or in the whole~$\h$, and we may assume~$\{\xikj\}\subset\Om$.}
According to {Step~4}, we shall distinguish two cases:

\noindent{\em Case I:} We have~${\rm dist}(\xikj,\partial\Om)={\rm O}(\lambdakj)$ as~$k\to\infty$.

\noindent{\em Case II:} We have~$\dfrac{{\rm dist}(\xikj,\partial\Om)}{\lambdakj}\to\infty$ as~$k\to\infty$ (This only happens when
$\xikj \subset\Om$ stays inside the domain or approaches the boundary  slower than~$\lambdakj$).

\vs
We know by~\eqref{conv_prof} that~$\ukj(\cdot):=u_k(\tau_{\xikj}(\delta_{\lambdakj}(\cdot)))\rightharpoonup\uj(\cdot)$ in~$\Sc(\h)$,
and that~${\rm supp} \, \ukj\,\subset \delta_{1/\lambdakj} (\tau_{\xikj^{-1}}(\Om)))$.
  Note that, because of the condition on the scaling parameters,
  \[\delta_{1/\lambdakj} (\tau_{\xikj^{-1}}(\Om)) \to \Om_{\rm o}^{(j)} \ \text{as} \
k\to\infty,
\]
where~$\Om_{\rm o}^{(j)}$ is either an open half-space or the entire space (in Case II).

Fix~$\varphi\in C_0^\infty(\Om_{\rm o}^{(j)})$. By the invariance
of the~$\Sc$ and the~$L^{2^*}$ norms with respect to the scaling
\begin{equation}\label{scal_inv}
	u(\cdot) \mapsto \widetilde{u}(\cdot)=\eta^{\frac{2-Q}{2}}u(\delta_{1/\eta}(\tau_{(\xi)^{-1}}(\cdot))),
\end{equation}
 it is clear that the rescaled function~$\widetilde{\varphi}(\cdot):=\lambdakj^{\frac{2-Q}{2}}\varphi(\delta_{1/\lambdakj}(\tau_{\xikj}^{-1}(\cdot)))\in C_0^\infty(\Om)$ and
\begin{equation}\label{thm_glob_comp_2}
\langle {\rm d}\El(u_k),\widetilde{\varphi}\rangle = \langle {\rm d}\Es(\ukj),\varphi\rangle-\lambda (\lambdakj)^{1+Q/2}\int_{\h}\ukj\varphi 	\,{\rm d}\xi.
\end{equation}
Now, by~\eqref{PS2},
\begin{equation*}
\langle {\rm d}\El(u_k),\widetilde{\varphi}\rangle=\lambdakj^{\frac{2-Q}{2}} \langle {\rm d}\El(u_k),\varphi\big(\delta_{1/\lambdakj}(\tau_{\xikj}^{-1}(\cdot))\big)\rangle\to 0 \quad\text{as }k\to\infty.
\end{equation*}
On the other hand, since~$\varphi$ is compactly supported,~$\ukj\to\uj$ strongly in ~$L^{2}_{\rm{loc}}(\h)$ and~$\lambdakj\to0$ by {Step 4} we get
\begin{equation}\label{thm_glob_comp_4}
(\lambdakj)^{1+Q/2}\int_{\h}\ukj\varphi 	\,{\rm d}\xi \to 0 \quad\text{as }k\to\infty.
\end{equation}
Combining~\eqref{thm_glob_comp_2}--\eqref{thm_glob_comp_4} with the weak continuity of~${\rm d}\Es$ we get~$\langle {\rm d}\Es(\uj),\varphi\rangle= 0$ for all~$\varphi\in C_0^\infty(\Om_{\rm o}^{(j)})$. Hence, the claim follows by density.

In remains to show that ~$\{\xikj\}\subset\Om$. As already observed, this is obvious in Case II, and so we only need to prove it in Case I, that is when~${\rm dist}(\xikj,\partial\Om)={\rm O}(\lambdakj)$ as~$k\to\infty$. In that case, fix~$\bar{\xi}^{(j)}\in \Om_{\rm o}^{(j)}$ and set
\begin{align*}
&\bar{\xi}_k^{(j)}:=\tau_{\xikj}(\delta_{\lambdakj}( \bar{\xi}^{(j)})),\\
&\bar{u}_k^{(j)}(\cdot)=u_k(\tau_{\bar{\xi}_k^{(j)}}(\delta_{\lambdakj}( \cdot))=u_k(\tau_{\xikj}(\delta_{\lambdakj}( \tau_{\bar{\xi}^{(j)}}(\cdot)))).
\end{align*}
Then,~$\{\bar{\xi}_k^{(j)}\}\subset\Om$ for~$k$ sufficiently large and~$\bar{u}_k^{(j)}(\cdot)\rightharpoonup\uj(\tau_{\bar{\xi}^{(j)}}(\cdot))$ in~$\Sc(\h)$ as~$k\to\infty$. Therefore, the claim follows by taking~$\bar{\xi}_k^{(j)}$ as translation parameter and~$\bar{u}^{(j)}(\cdot)=\uj(\tau_{\bar{\xi}^{(j)}}(\cdot))$.

\vspace{2mm}
\noindent{\bf Step 6. The profiles~$\uj$ are in finite number.} By~\eqref{prof2} of Theorem~\ref{thm_prof}, it suffices to show that the~$
\Sc$ norms of the nontrivial profiles~$\uj$,~$j\in\mathbb{N}$, are uniformly bounded from below. To this aim, let us test~\eqref{pzero} with~$\varphi=\uj$. Thus, from the Folland-Stein inequality~\eqref{folland}, we get
\[
\|\uj\|_{\Sc}^{2}\,=\,\|\uj\|_{L^{2^*}}^{2^*}
\,\leq \, S^* \big(\|\uj\|_{\Sc}^{2}\big)^{\frac{2^*}{2}}.\]
Consequently,~$\|\uj\|_{\Sc}\geq (S^*)^{\frac{2}{2-2^*}}$ and so the claim is proved.

\vspace{2mm}
\noindent{\bf Step 7. The sequence of remainders~$\{\rkj\}$ strongly converges to 0 in~$\Sc(\h)$.}
First, by the scaling invariance we can easily see that
\begin{eqnarray}\label{thm_glob_comp_5}
\|\rkj\|_{\Sc}^2
&=&\|u_k-\uz\|_{\Sc}^2+\sum_{j=1}^{J}\left\|\lambdakj^{\frac{2-Q}{2}}\uj(\delta_{1/\lambdakj}(\tau_{\xikj}^{-1}(\cdot)))\right\|_{\Sc}^2
\notag\\*[0.5ex]
&&
-2\left(u_k-\uz, \sum_{j=1}^{J} \lambdakj^{\frac{2-Q}{2}}\uj(\delta_{1/\lambdakj}(\tau_{\xikj}^{-1}(\cdot)))\right) +o(1)\notag\\*[0.5ex]
&=&\|u_k-\uz\|_{\Sc}^2+\sum_{j=1}^{J}\|\uj\|_{\Sc}^2-2\sum_{j=1}^{J}(\ukj,\uj)\notag\\*[0.5ex]
&& +2\sum_{j=1}^{J}\left( \uz, \lambdakj^{\frac{2-Q}{2}}\uj(\delta_{1/\lambdakj}(\tau_{\xikj}^{-1}(\cdot)))\right)+o(1)\notag\\*[0.5ex]
&=&\|u_k-\uz\|_{\Sc}^2-\sum_{j=1}^{J}\int_{\h} |{u}^{(j)}|^{2^*}	\,{\rm d}\xi +o(1),
\end{eqnarray}
where we used the fact that~$\uj$ solves~\eqref{pzero} by Step 5 for any~$j\in\Irm$.
Now, arguing as in {Step 3}, as~$k\to\infty$
\begin{equation*}
\|u_k-\uz\|^{2}_{\Sc(\Omega)}=  \int_{\Om}\left(|u_k|^{2^*-2}u_k- |\uz|^{2^*-2}\uz\right)\left(u_k-\uz\right)	\,{\rm d}\xi=o(1),
\end{equation*} 
{and, 
using again~\eqref{thm_glob_comp_conv}, the previous equality becomes, as~$k\to\infty$,}
\begin{equation*}
\|u_k-\uz\|^2_{\Sc(\Omega)}= \|u_k\|_{L^{2^*}(\Om)}^{2^*}- \|\uz\|_{L^{2^*}(\Om)}^{2^*}+o(1),
\end{equation*}
which yields in turn, thanks to~\eqref{thm_glob_comp_5}
\begin{equation}\label{thm_glob_comp_6}
\lim_{k\to\infty}\|\rkj\|_{\Sc}^2=\lim_{k\to\infty}\|u_k\|_{L^{2^*}(\Om)}^{2^*}- \|\uz\|_{L^{2^*}(\Om)}^{2^*}
-\sum_{j=1}^{J}\| {u}^{(j)} \|_{L^{2^*}(\h)}^{2^*}.
\end{equation}
Hence, we are done if we show that
\begin{equation}\label{eq:critical-limit-norm}
\lim_{k\to\infty}\|u_k\|_{L^{2^*}(\Om)}^{2^*}= \|\uz\|_{L^{2^*}(\Om)}^{2^*}
+\sum_{j=1}^{J}\| {u}^{(j)} \|_{L^{2^*}(\h)}^{2^*}.
\end{equation}
By~\eqref{prof1}, we know that~$u_k(\cdot)=\sum_{j=1}^{\ell}\lambdakj^{\frac{-Q}{2^*}}\psi^{(j)}(\delta_{1/\lambdakj}(\tau_{\xikj}^{-1}(\cdot)))+\rkj(\cdot)$. Thus, by using
iteratively the Brezis-Lieb lemma and the
invariance of the~$L^{2^*}$ norm with respect to~\eqref{scal_inv} together with~\eqref{prof3}, we obtain the desired conclusion.

\vspace{2mm}
\noindent{\bf Step 8. Completion of the proof.}
Clearly,~\eqref{propr2} follows straightly from~\eqref{strange}, while and~\eqref{propr1} and~\eqref{propr3} come from~\eqref{prof2}, thanks to Step 7. Hence, it
remains to show~\eqref{propr4}. To this end, let us recall that by~\eqref{thm_glob_comp_conv},~$u_k \to \uz$ in~$L^2(\Om)$, where and~$\uz$ solves~\eqref{plambda} thanks to Step 2. Moreover, because of Step 5, we know that
$\|\uj\|_{\Sc}^{2}=\|\uj\|_{L^{2^*}}^{2^*}$ . Combining these two facts with~\eqref{propr3},~\eqref{eq:critical-limit-norm} and
the definition of~$\El$, we obtain the validity of ~\eqref{propr4}. This completes the proof.
\end{proof}
\vspace{2mm}

It is now worth remarking that in the special case when further regularity is assumed on the set $\Om$ we can give a complete characterization of the limiting sets $\{\Om_{\rm o}^{(j)}\}_{j \in \Irm}$. Let us recall the definition of $H$-flatness of $\Om$ at a boundary point.

\begin{defn}[{\bf $H$-flat domains}]\label{H_flat}
    Let $\Om$ be a smooth bounded domain of~$\h$ and let $\xi \in \partial \Om$. Assume $\varPhi \in C^\infty(\h)$ be a defining function for the boundary of $\Om$ in a neighborhood of $\xi$ \big(that is, $\varPhi : B_\rr(\xi) \to \r$ with $\varPhi \equiv 0$ on $B_\rr(\xi) \cap \partial \Om$ and $\varPhi >0$ on $B_\rr(\xi) \cap \Om$\big). {The point~$\xi$ is called characteristic if~$D_H \varPhi(\xi)=0$}. Moreover, if $\xi$ is a characteristic point, then we say that $\Om$ is $H$\textup{-flat} at $\xi$ if
    \[
    q_H\varPhi (\xi)=0,
    \]
    where $q_H: \r^{2n} \to \r$ is the quadratic form associated with the Hessian matrix $D^2_H \varPhi$, that is
    \[
    \big(q_H\varPhi(\eta) \big)(z):= \sum_{i,j=1}^{2n} \big((D_H)_i(D_H)_j\varPhi\big)(\eta) z_iz_j \qquad \forall z \in \r^{2n}.
    \]
    We say that $\Om$ is $H$-flat if it is $H$-flat at any characteristic point of its boundary.
\end{defn}

We have the following
  \begin{lemma}[Lemma 3.4 in \cite{CU01}]\label{citti_uguzzoni}
      Let $\Om$ be a smooth bounded domain in $\h$ and, for any sequence $\{\xi_k\}$ in $\Om$ and for any divergent series $\{\lambda_k\} \in \r^+$, let
      \[
      \Om_k := \delta_{\lambda_k}(\tau_{\xi_k^{-1}}(\Omega)).
      \]
      Then, up to subsequences, $\Om_k \to \Om_{\rm o}$ where
      \begin{enumerate}[\rm(i)]
          \item if $\lambda_k {\rm dist}(\xi_k,\partial \Om)$ is unbounded then $\Om_{\rm o}$ is the whole space $\h$;
          \item if $\lambda_k {\rm dist}(\xi_k,\partial \Om)$ is bounded, then there exists $\bm{a}$,~$\bm{b} \in \r^n$ and~$\bm{c} \in \r$ and a quadratic form $q$ on $\r^{2n}$ such that, up to translations,
          \[
          \Om_{\rm o} = \big\{(x,y,t) \in \h \, | \, \bm{a} \cdot x + \bm{b} \cdot y +  \bm{c} \cdot t + q(x,y) >0\big\};
          \]
          \item if $\lambda_k {\rm dist}(\xi_k,\partial \Om)$ is bounded and $\Om$ is $H$-flat, according to {\rm Definition~{\rm \ref{H_flat}}}, then $\Om_{\rm o}$ is a half-space of $\h$.
      \end{enumerate}
  \end{lemma}

\vspace{2mm}

The main consequence of Lemma \ref{citti_uguzzoni} is that for bounded $C^\infty$ domains in $\h$ we explicitly know the limiting sets $\Om_{\rm o}^{(j)}$. Moreover, if we consider in Theorem~{\rm\ref{thm_glob_comp}} only positive Palais-Smale sequences for $\El$ on $H$-flat domains, then we have that the limiting spaces~$\Om_{\rm o}^{(j)}$ are always the whole space $\h$. This follows from the nonexistence results proven in \cite{LU98}.

 \vspace{3mm}


\vspace{2mm}

\begin{thebibliography}{99}

		
\bibitem{AG03} {\sc M. Amar, A. Garroni}: {$\Gamma$-convergence of concentration problems}. {\it Ann. Scuola Norm. Sup. Pisa Cl. Sci.} {\bf 2}~(2003), 151--179.
\vs




\bibitem{Ben08}{\sc J.~Benameur}: {Description du d\'efaut de compacit\'e de l'injection de Sobolev sur le groupe de Heisenberg}. {\it Bull. Belg. Math. Soc. Simon Stevin}~{\bf 15} (2008), no.~4, 599--624.
\vs
		
		
	
		
	
\bibitem{BFM13} {\sc T. Branson, L. Fontana, C. Morpurgo}:
Moser-Trudinger and Beckner-Onofri's inequalities on the CR sphere. {\it Ann. of Math.} {\bf 177} (2013), no.~1, 1--52.
\vs


\bibitem{BP89}{\sc H. Brezis, L.~\!A. Peletier}: Asymptotic for Elliptic Equations involving critical growth. In {\it Partial Differential Equations and the Calculus of Variations. Essays in Honor of Ennio De Giorgi, Vol. 1}, Progr. Differ. Equ. Appl., Birkh\"auser, Boston (1989), 149--192.
\vs



	\bibitem{CLMR23}{\sc G. Catino, Y. Li, D. Monticelli, A. Roncoroni}: A Liouville theorem in the Heisenberg group. {\it Preprint} (2023). Available at~\url{https://arxiv.org/abs/2310.10469}.
\vs



\bibitem{Cit95} {\sc G. Citti}:  Semilinear Dirichlet problem involving critical exponent for the Kohn Laplacian. {\it Ann. Mat. Pura Appl.} (4) {\bf 169}~(1995), 375--392.
\vs



\bibitem{CU01}{\sc G. Citti, F. Uguzzoni}: Critical semilinear equations on the Heisenberg group: the effect of the topology of the domain. {\it Nonlinear Anal.}~{\bf 46}~(2001), 399--417.
\vs





     \bibitem{DGP07} {\sc D. Danielli, N. Garofalo, A. Petrosyan}: The sub-elliptic obstacle problem: $C^{1,\alpha}$ regularity of the free boundary in Carnot groups of step two. {\it Adv. Math.} {\bf 211} (2007), no. 2, 485--516.
    \vs


\bibitem{DPMP10} {\sc M.~del Pino, M.~Musso, F.~Pacard}:
Bubbling along boundary geodesics near the second critical exponent. {\it J. Eur. Math. Soc. (JEMS)} {\bf 12}~(2010), No.~6, 1553--1605.
\vs

			\bibitem{DMM19} {\sc S.~Deng, F.~Mahmoudi, M.~Musso}:
Concentration at sub-manifolds for an elliptic Dirichlet problem near high critical exponents. {\it
Proc. Lond. Math. Soc.}~{\bf 118}(2019), No.~3, 379--415.
			\vs



\bibitem{FM99}{\sc M. Flucher, S. M\"uller}: Concentration of low energy extremals. {\it Ann. Inst. H. Poincar\'e - Anal. Non Lin\'eaire}~{\bf 16} (1999), no.~3, 269--298.
\vs


\bibitem{FV23}{\sc J. Flynn, J. Vétois}: Liouville-type results for the CR Yamabe equation in the Heisenberg group. {\it Preprint} (2023). Available at~\url{https://arxiv.org/abs/2310.14048}.
\vs

 \bibitem{FS74}{\sc   G.~\!B.~Folland, E.~\!M. Stein}: Estimates for the $\bar\partial_b$ complex and analysis on the Heisenberg group. {\it Comm. Pure Appl. Math.}~{\bf 27} (1974), 429--522.
 \vs
	



\bibitem{FGMT15} {\sc R.~\!L. Frank, M.~\!d.~\!M. Gonz\'alez, D.~\!D. Monticelli, J. Tan}: An extension problem for the CR fractional Laplacian. {\it Adv. Math.}~{\bf 270} (2015), 97--137.
\vs

\bibitem{FL12}{\sc R. Frank, E.~\!H. Lieb}: Sharp constants in several inequalities on the Heisenberg group. {\it Ann. Math.}~{\bf 176}~(2012), 349--381.
\vs
 

 \bibitem{GL92}{\sc N. Garofalo, E. Lanconelli}: Existence and Nonexistence Results for Semilinear Equations on the Heisenberg Group. {\it Indiana Univ. Math. J.} {\bf 41}~(1992), no.~1, 71--98.

\vs
 


  \bibitem{GV00}{\sc N. Garofalo, D. Vassilev}: Regularity near the characteristic set
in the non-linear Dirichlet problem
and conformal geometry of sub-Laplacians on Carnot groups. {\it Math. Ann.} {\bf 318}~(2000), 453--516.
\vs

  \bibitem{GV01}{\sc N. Garofalo, D. Vassilev}: Symmetry properties of positive entire solutions of Yamabe-type equations on groups of Heisenberg type. {\it Duke Math. J.} {\bf 106}~(2001), 411--448.
\vs



   \bibitem{Ger98} {\sc P. G\'erard}: Description du d\'efaut de compacit\'e de l'injection de Sobolev. {\it ESAIM Control Optim. Calc.
   Var.} {\bf 3}~(1998), 213--233.
   \vs

\bibitem{GMM18}{\sc C. Guidi, A. Maalaoui, V. Martino}: Palais-Smale sequences for the fractional CR Yamabe functional and multiplicity results. {\it Calc. Var. Partial Differential Equations}~{\bf 57} (2018), Art.~152.
\vs




  \bibitem{Han91} {\sc Z.-C. Han}: {Asymptotic approach to singular solutions for nonlinear elliptic equations involving critical Sobolev exponent}. {\it Ann. Inst. H. Poincar\'e - Anal. Non Lin\'eaire} {\bf 8}~(1991),  159--174.
  	\vs	

\bibitem{IV11}{\sc S.~\!P. Ivanov, D.~\!B. Vassilev}: {\it Extremals for the Sobolev Inequality and the Quaternionic Contact Yamabe Problem}. World Scientific Publishing, Singapore, 2011.

\vs 


\bibitem{JL88} {\sc D. Jerison, J.~\!M. Lee}: Extremals for the Sobolev inequality on the Heisenberg group and the~CR~Yamabe problem. {\it J. Amer. Math. Soc.}~{\bf 1}~(1988), no.~1, 1--13.
\vs


\bibitem{LL12} {\sc N. Lam, G. Lu}:
Sharp Moser-Trudinger inequality on the Heisenberg group at the critical case and applications.
{\it Adv. Math.}~{\bf 231} (2012), no.~6, 3259--3287.
\vs




\bibitem{LU98} {\sc E. Lanconelli, F. Uguzzoni}:
Asymptotic behaviour and non-existence theorems for semilinear Dirichlet problems involving critical exponent on unbounded domains of the Heisenberg group.
{\it Boll. UMI (Serie 8)} {\bf 1-B}~(1998), no.~1, 139--168.
\vs




\bibitem{MMP13} {\sc A. Maalaoui, V. Martino, A. Pistoia}: Concentrating solutions for a sub-critical sub-elliptic problem. {\it Diff. Int. Eq.}~{\bf 26}~(2013), no. 11-12, 1263--1274.
\vs




						\bibitem{MPPP23} {\sc M. Manfredini, G. Palatucci, M. Piccinini, S. Polidoro}:
		{H\"older continuity and boundedness estimates for nonlinear fractional equations in the    Heisenberg group}. {\it J. Geom. Anal.}~{\bf 33}~(2023), no.~3, Art.~77.
		\vs
				

    \bibitem{Pal11} {\sc G. Palatucci}: {Subcritical approximation of the Sobolev quotient and a related concentration result}. {\it Rend. Sem. Mat. Univ. Padova} {\bf 125}~(2011), 1--14.
    \vs

    \bibitem{Pal11b} {\sc G. Palatucci}: {$p$-Laplacian problems with critical Sobolev exponent}. {\it Asymptot. Anal.} {\bf 73} (2011), 37--52.
    \vs

 			
		\bibitem{PP22} {\sc G. Palatucci, M. Piccinini}:
		{Nonlocal Harnack inequalities in the Heisenberg group}.
		{\it Calc. Var. Partial Differential Equations}~{\bf 61} (2022), Art.~185.
		\vs

		\bibitem{PP23}{\sc G. Palatucci, M. Piccinini}:
		Nonlinear fractional equations in the Heisenberg group. {\it Bruno Pini Math. Anal. Semin.} {\bf 14} (2023), no.~2, 163--200.
\vs

		\bibitem{PP23h} {\sc G. Palatucci, M. Piccinini}:
		{Asymptotic approach to singular solutions for the CR Yamabe equation}.	{\it Preprint} (2024).
		\vs
		
		\bibitem{PPT23} {\sc G. Palatucci, M. Piccinini, L. Temperini}:
		{Global Compactness, subcritical approximation of the Sobolev quotient, and a related concentration result in the Heisenberg group}. 
		{\it Trends Math.}~{\bf 3} (2024),  145--155.
		\vs

		\bibitem{PP14} {\sc G. Palatucci, A. Pisante}: Improved Sobolev embeddings, profile decomposition, and concentration-compactness for fractional Sobolev spaces. {\it Calc. Var. Partial Differential Equations} {\bf 50}~(2014), no.~3-4, 799--829.
		\vs
		
\bibitem{PP15} {\sc G. Palatucci, A. Pisante}: A Global Compactness type result for Palais-Smale sequences in fractional Sobolev spaces.
{\it Nonlinear Anal.} {\bf 117}~(2015), 1--7.
\vs


\bibitem{PPS15} {\sc G. Palatucci, A. Pisante, Y. Sire}: Subcritical approximation of a Yamabe-type nonlocal equation: a Gamma-convergence approach.
{\it Ann. Sc. Norm. Super. Pisa Cl. Sci.~(5)}~{\bf 14}~(2015), No. 3, 819--840.
\vs

\bibitem{Pas93}{\sc D.  Passaseo}: Nonexistence results for elliptic problems with supercritical nonlinearity in nontrivial domains. {\it J. Funct. Anal.}  {\bf 114} (1993), 97--105.
\vs


	\bibitem{PV24}{\sc J.V. Prajapat, A.S. Varghese}: Symmetry and classification of solutions to an integral equation in the Heisenberg group $\mathbb{H}^n$. {\it Preprint} (2023). Available at~\url{https://arxiv.org/abs/2401.15100}.
\vs


\bibitem{PT22} {\sc P. Pucci, L. Temperini}: On the concentration-compactness principle for Folland-Stein spaces and for fractional horizontal Sobolev spaces. {\it Math. Eng.}~{\bf 5} (2022), No.~1, 1--21.
\vs

		\bibitem{Rey89} {\sc O. Rey}: {Proof of the conjecture of H.~Brezis and L.~\!A.~Peletier}. {\it Manuscripta math.} {\bf 65} (1989),  19--37.
					\vs
			

					
		\bibitem{Str84} {\sc M. Struwe}: A global compactness result for elliptic boundary value problems involving limiting nonlinearities. {\it Math. Z.} {\bf 187} (1984), 511--517.	
		\vs
		
		\bibitem{Tin20}
		{\sc K. Tintarev}: {\it Concentration Compactness: Functional-Analytic Theory of Concentration Phenomena}. De Gruyter Series in Nonlinear Analysis and Applications Book 33, 2020.
\vs

\bibitem{TF07} {\sc K. Tintarev, K.\!~-H. Fieseler}: {\it Concentration Compactness}. Functional-Analytic Grounds and Applications.
Imperial College Press, London (2007). vii+265pp.
\vs


\bibitem{U99} {\sc F. Uguzzoni}: A non-existence theorem for a semilinear Dirichlet problem involving critical exponent on halfspaces of the Heisenberg group. {\it NoDEA Nonlinear Differential Equations Appl.} {\bf 6}  (1999), 191--206.
 \vs
 
 \bibitem{Vas06} {\sc D. Vassilev}: Existence of solutions and regularity near the characteristic boundary for sub-Laplacian equations on Carnot groups. {\it Pacific J. Math.} {\bf 227} (2006), no. 2, 361--397.
 \vs


\end{thebibliography}
\end{document}